\documentclass[10pt]{amsproc}
\usepackage[utf8]{inputenc}
\usepackage[dvipsnames]{xcolor}
\usepackage{todonotes}

\usepackage{etoolbox}

\makeatletter
\patchcmd{\@thm}{\let\thm@indent\indent}{\let\thm@indent\noindent}{}{}
\patchcmd{\@thm}{\thm@headfont{\scshape}}{\thm@headfont{\bfseries}}{}{}

\usepackage[T1]{fontenc}

\usepackage{mathtools}

\usepackage{multirow}
\usepackage{longtable}

\usepackage{amsmath}
\usepackage{amsfonts}
\usepackage{amssymb}
\usepackage{float}
\usepackage{amsthm}
\usepackage{comment}
\usepackage{amscd}
\usepackage{amsxtra}
\usepackage{epsfig}
\usepackage{epigraph}
\usepackage{pgfplots}
\usepackage{enumitem}

\usepackage{stmaryrd}

\usepackage{listings}
\usepackage{color}

\definecolor{dkgreen}{rgb}{0,0.6,0}
\definecolor{gray}{rgb}{0.5,0.5,0.5}
\definecolor{mauve}{rgb}{0.58,0,0.82}

\usepackage{color}
\usepackage[all]{xy}
\usepackage{verbatim}

\usepackage{eucal}
\usepackage{mathrsfs}

\usepackage{graphicx}
\usepackage{tikz-cd}

\usepackage{url}
\usepackage{verbatim}

\usepackage{xy}
\xyoption{all}
\newcommand{\xar}[1]{\xrightarrow{{#1}}}

\usepackage{amsthm}
\usepackage{hhline}

\usepackage{caption} 
\makeatletter
\def\@tocline#1#2#3#4#5#6#7{\relax
  \ifnum #1>\c@tocdepth 
  \else
    \par \addpenalty\@secpenalty\addvspace{#2}%
    \begingroup \hyphenpenalty\@M
    \@ifempty{#4}{%
      \@tempdima\csname r@tocindent\number#1\endcsname\relax
    }{%
      \@tempdima#4\relax
    }%
    \parindent\z@ \leftskip#3\relax \advance\leftskip\@tempdima\relax
    \rightskip\@pnumwidth plus4em \parfillskip-\@pnumwidth
    #5\leavevmode\hskip-\@tempdima
      \ifcase #1
       \or\or \hskip 1em \or \hskip 2em \else \hskip 3em \fi%
      #6\nobreak\relax
    \hfill\hbox to\@pnumwidth{\@tocpagenum{#7}}\par
    \nobreak
    \endgroup
  \fi}
\makeatother

\usepackage[noabbrev,capitalize]{cleveref} 

\crefformat{nul}{(#2#1#3)}
\Crefformat{nul}{(#2#1#3)}

\crefname{section}{\S}{\S\S}
\crefname{subsection}{\S}{\S\S}
\crefname{axioms}{Axiom}{Axioms}
\crefname{exercise}{Exercise}{Exercises}
\crefname{exercisenum}{Exercise}{Exercises}
\crefname{construction}{Construction}{Constructions}
\crefname{problem}{Problem}{Problems}
\crefname{theorem}{Theorem}{Theorems}
\crefname{definition}{Definition}{Definitions}
\crefname{prop}{Proposition}{Propositions}
\crefname{lemma}{Lemma}{Lemmas}
\crefname{example}{Example}{Examples}
\crefname{examplealph}{Example}{Examples}
\crefname{corollary}{Corollary}{Corollaries}
\crefname{nonexample}{Nonexample}{Nonexamples}
\crefname{equation}{}{}
\crefname{summary}{Summary}{Summaries}
\crefname{recollection}{Recollection}{Recollections}
\Crefname{recollection}{Recollection}{Recollections}
\Crefname{nonexample}{Nonexample}{Nonexamples}
\Crefname{corollary}{Corollary}{Corollaries}
\Crefname{corollary}{Corollary}{Corollaries}
\Crefname{axioms}{Axiom}{Axioms}
\Crefname{exercise}{Exercise}{Exercises}
\Crefname{exercisenum}{Exercise}{Exercises}
\Crefname{construction}{Construction}{Constructions}
\Crefname{problem}{Problem}{Problems}
\Crefname{theorem}{Theorem}{Theorems}
\Crefname{definition}{Definition}{Definitions}
\Crefname{prop}{Proposition}{Propositions}
\Crefname{lemma}{Lemma}{Lemmas}
\Crefname{example}{Example}{Examples}
\Crefname{examplealph}{Example}{Examples}
\Crefname{section}{\S}{\S\S}
\Crefname{subsection}{\S}{\S\S}

\DeclareMathOperator{\Hom}{Hom}

\DeclareMathOperator{\spec}{Spec}
\DeclareMathOperator{\spf}{Spf}

\DeclareMathOperator{\coker}{coker}

\newtheorem{lemma}{Lemma}[subsection]

\newtheorem{corollary}[lemma]{Corollary}

\newtheorem{theorem}[lemma]{Theorem}

\newtheorem*{defno}{Definition}

\newtheorem{prop}[lemma]{Proposition}
\newtheorem{conjecture}[lemma]{Conjecture}

\newtheorem{abc}{abc}
\newtheorem{ABCthm}[abc]{Theorem}

\newtheorem*{conj-moore}{Conjecture~\ref{moore-splitting}}

\theoremstyle{definition}

\newtheorem{question}[lemma]{Question}

\newtheorem{definition}[lemma]{Definition}
\newtheorem{warning}[lemma]{Warning}
\newtheorem{example}[lemma]{Example}

\newtheorem{remark}[lemma]{Remark}
\newtheorem{notation}[lemma]{Notation}

\newtheorem{recall}[lemma]{Recollection}

\newtheorem{observe}[lemma]{Observation}

\renewcommand{\AA}{\mathbf{A}}
\newcommand{\FF}{\mathbf{F}}
\newcommand{\Z}{\mathbf{Z}}
\newcommand{\QQ}{\mathbf{Q}}
\newcommand{\cc}{\mathbf{C}}

\newcommand{\cd}{\mathscr{D}}

\newcommand{\co}{\mathscr{O}}
\newcommand{\ce}{\mathscr{E}}

\newcommand{\GG}{\mathbf{G}}

\newcommand{\cL}{\mathscr{L}}

\newcommand{\Eoo}{{\mathbf{E}_\infty}}

\newcommand{\CAlg}{\mathrm{CAlg}}

\newcommand{\univ}{\mathrm{univ}}

\newcommand{\ev}{\mathrm{ev}}

\newcommand{\ul}[1]{\underline{#1}}

\newcommand{\id}{\mathrm{id}}

\newcommand{\KU}{\mathrm{KU}}

\newcommand{\ku}{\mathrm{ku}}

\renewcommand{\H}{\mathrm{H}}

\newcommand{\F}{\mathrm{F}}

\newcommand{\HC}{\mathrm{HC}}
\newcommand{\HP}{\mathrm{HP}}

\newcommand{\Sch}{\mathrm{Sch}}
\newcommand{\op}{\mathrm{op}}

\newcommand{\dR}{\mathrm{dR}}

\newcommand{\gr}{\mathrm{gr}}

\newcommand{\can}{\mathrm{can}}

\renewcommand{\log}{\mathrm{log}}

\renewcommand{\S}{Section }

\makeatletter
\providecommand{\leftsquigarrow}{%
  \mathrel{\mathpalette\reflect@squig\relax}%
}
\newcommand{\reflect@squig}[2]{%
  \reflectbox{$\m@th#1\rightsquigarrow$}%
}
\makeatother

\newcommand{\lcm}{\mathrm{lcm}}

\newcommand{\bull}{\bullet}

\newcommand{\pw}[1]{[\![#1]\!]}

\renewcommand{\tilde}{\widetilde}

\renewcommand{\min}{\mathrm{min}}

\newcommand{\Gr}{\mathrm{Gr}}

\newcommand{\rot}{\mathrm{rot}}

\usepackage{slashed}

\usepackage{relsize}
\usepackage[bbgreekl]{mathbbol}
\usepackage{amsfonts}
\DeclareSymbolFontAlphabet{\mathbb}{AMSb} 
\DeclareSymbolFontAlphabet{\mathbbl}{bbold}

\newcommand{\ptl}{{\tilde{p}}}

\newcommand{\pdb}[1]{\langle{#1}\rangle}

\newcommand{\sdr}[1]{s\Omega_{\square,#1}}
\newcommand{\fdr}[1]{F\Omega_{\square,#1}}
\newcommand{\Gammaf}[1]{\Gamma_{F\text{-}\mathrm{dR}}({#1})}
\newcommand{\Poly}{\mathrm{Poly}}
\newcommand{\fdiff}[1]{F\cd_{\square,#1}}

\newcommand{\plog}{\mathrm{Li}}

\newcommand{\ps}[1]{[\![#1]\!]}
\newcommand{\ab}[1]{\langle #1 \rangle}

\mathchardef\mhyphen="2D

\newcommand{\combbinom}[4]{\genfrac{(}{|}{0pt}{}{#1}{#2}\!\!\genfrac{|}{)}{0pt}{}{#3}{#4}}

\newcommand{\bigmid}{\;\middle|\;}

\title{Generalized $n$-series and de Rham complexes}
\author{S. K. Devalapurkar and M. L. Misterka}
\thanks{Part of this work was done when the first author was supported by the PD Soros Fellowship, Inpher, and NSF DGE-2140743}

\begin{document}

\maketitle

\begin{abstract}
    The goal of this article is to study some basic algebraic and combinatorial properties of ``generalized $n$-series'' over a commutative ring $R$, which are functions $s: \Z_{\geq 0} \to R$ satisfying a mild condition. A special example of generalized $n$-series is given by the $q$-integers $\frac{q^n-1}{q-1} \in \Z\pw{q-1}$.
    Given a generalized $n$-series $s$, one can define $s$-analogues of factorials (via $n!_s = \prod_{i=1}^n s(n)$) and binomial coefficients. We prove that Pascal's identity, the binomial identity, Lucas' theorem, and the Vandermonde identity admit $s$-analogues; each of these specialize to their appropriate $q$-analogue in the case of the $q$-integer generalized $n$-series.
    We also study the growth rates of generalized $n$-series defined over the integers.
    Finally, we define an $s$-analogue of the ($q$-)derivative, and prove $s$-analogues of the Poincar\'e lemma and the Cartier isomorphism for the affine line, as well as a pullback square due to Bhatt-Lurie.
\end{abstract}

\tableofcontents

\newpage
\section{Introduction}

\subsection{Summary}

Recent work of Bhatt, Drinfeld, Lurie, Morrow, Scholze, and others (see, e.g., \cite{bms-i, bhatt-scholze, scholze-q-def, drinfeld-formal-group, drinfeld-prism, apc}) has shown that $q$-deformations of classical number-theoretic and algebro-geometric concepts play a central role in arithmetic geometry. 
The basic premise behind the theory of $q$-deformations is the idea that the $q$-integers $[n]_q = \frac{q^n - 1}{q-1}$ display many similarities to the ordinary integers. This idea has a rich history\footnote{Any source recounting the history of $q$-deformations would be most welcome to the authors!}: the basic ideas date back at least to Euler (e.g., \cite{original-euler}) and Jacobi and the study of basic hypergeometric series. A $q$-analogue of the derivative originates with Jackson in 1909 (see \cite{original-jackson}). We refer the reader to the book \cite{quantum_calculus} for an exposition of $q$-deformed calculus.

The theory of \textit{formal group laws} supplies a simultaneous generalization of both ordinary integers and $q$-integers (see \cref{recall: fgl-recall} for a quick summary of the basics of formal group laws, and \cite{hazewinkel, green} for a detailed treatment). Namely, every formal group law $F$ over a ring $R$ defines a sequence of power series $\pdb{n}$ over $R$ for every integer $n\in \Z$. In the case of the the additive formal group law $x + y$, we have $\pdb{n} = n$; and in the case of the multiplicative formal group law $x + y + xy$, one can identify $\pdb{n} = [n]_q$. The goal of this article is to explore whether certain aspects of $q$-deformed mathematics (such as $q$-analogues of basic combinatorial formulae, and properties of the $q$-de Rham complex of \cite{scholze-q-def}) admit generalizations to arbitrary formal group laws. One of the primary motivations behind our investigation is the unpublished observation of Arpon Raksit that homotopy-theoretic methods naturally suggest studying ``$F$-analogues'' of the $q$-de Rham complexes arising in the aforementioned work of Bhatt-Morrow-Scholze, as well as the calculation of \cite[Section 3.3]{coh-gr}.\footnote{Since the actual homotopy theory does not play any role in this paper, we refer the interested reader to \cref{rmk: homotopy theory} below for more.}

Our primary observation is that one does not need the structure of a formal group law to define and study these ``$F$-analogues''. Instead, the following significantly weaker structure suffices:
\begin{defno}[\cref{def: generalized n-series}]
Fix a ring $R$ (always assumed commutative with unit). A \textit{generalized $n$-series (GNS) over $R$} is a function $s : \Z_{\ge 0} \to R$ such that:
\begin{enumerate}
    \item $s(0) = 0$,
    \item $s(n)$ is not a zero-divisor for any $n > 0$,
    \item $s(n - k) \mid s(n) - s(k)$ for all $n > k > 0$.
\end{enumerate}
\end{defno}
For instance, the map $s: \Z_{\geq 0} \to \Z\pw{q-1}$ sending $n\mapsto [n]_q = \frac{q^n-1}{q-1}$ defines a GNS.

In the body of this article, we show that this simple definition is sufficient for proving several analogues of classical combinatorial identities, and is also enough to study an ``$s$-deformation'' of the classical algebraic de Rham complex. The results of this article do not rely on any sophisticated tools: rather, the purpose is to demonstrate the efficiency of \cref{def: generalized n-series}. The work done in this article seems closely related to Bhargava's \cite{bhargava-factorial-AMM}, but we have not attempted to make a comparison.

If $n\geq 0$, let $n!_s = \prod_{k=1}^n s(k)$, and let $\binom{n}{j}_s = \frac{n!_s}{j!_s (n-j)!_s}$ denote the $s$-analogues of the factorial and binomial coefficient, respectively. Our main combinatorial results are the following; for the full statement of some of these results, we refer the reader to the body of the text.
\begin{ABCthm}
Fix a GNS $s$ over $R$.
The following hold:
\begin{enumerate}
    \item Pascal's identity (\cref{thm: s-Pascal identity}):
    $$\binom{n}{k}_s = \binom{n - 1}{k - 1}_s + \frac{s(n) - s(k)}{s(n - k)} \binom{n - 1}{k}_s.$$
    \item An $s$-analogue of the binomial and $q$-binomial theorems; see \cref{thm: s-binomial theorem}.
    \item Lucas' theorem (\cref{thm: s-Lucas theorem}): suppose that $s(1) = 1$ and
    $$s(a + b) \equiv s(a) + s(b) \pmod{s(a) s(b)}$$
    for all $a, b \in \Z_{> 0}$. Then, for any prime $p$ and any nonnegative integers $n_1, n_0, k_1, k_0$ such that $n_0, k_0 < p$, we have
    $$\binom{n_1 p + n_0}{k_1 p + k_0}_s \equiv \binom{n_1}{k_1} \binom{n_0}{k_0}_s \pmod{s(p)}.$$
    \item An analogue of the Vandermonde and $q$-Vandermonde identities; see \cref{thm: s-Vandermonde identity}.
\end{enumerate}
\end{ABCthm}
In \cref{sec: integer generalized n-series}, we study the growth rate of generalized $n$-series over $\Z$. For instance, we show in \cref{thm: lower bound on GNS over Z} that if $s(n)$ is a strictly increasing generalized $n$-series over $\Z$ which is not a scalar multiple of $n \mapsto [n]_q$ for any $q \in \Z_{> 0}$, then $s(n) = \Omega_a(a^n)$ for all $a \ge 0$.

As one might expect given our motivation above, one important class of examples of generalized $n$-series arises via formal group laws. Recall (see \cref{recall: fgl-recall}) that a formal group law over a commutative ring $R$ is a two-variable power series $x +_F y\in R\pw{x,y}$ such that $(x+_F y)+_F z = x+_F (y+_F z)$ and $x+_F y \equiv x + y \pmod{(x,y)^2}$. If $n\geq 0$ is an integer, the $n$-series of $F$ is defined via the formula
$$[n]_F(t) = \overbrace{t +_F t +_F \cdots +_F t}^n\in tR\pw{t}.$$
Let $\pdb{n}_F = \frac{[n]_F(t)}{t}$. Suppose (for simplicity) that $R$ is torsionfree. Then, the function $\Z_{\geq 0} \to R\pw{t}$ sending $n\mapsto \pdb{n}_F$ defines a GNS over $R\pw{t}$ (\cref{prop: [n]_F(t) GNS}).

One can define the \textit{$F$-de Rham complex} of the affine line $\AA^1 = \spec R[x]$ as the cochain complex
$$\fdr{\AA^1} = \left(R\pw{t}[x] \xar{\nabla_F} R\pw{t}[x] dx\right), \ x^n \mapsto \pdb{n}_F x^{n-1} dx.$$
This was first defined by Arpon Raksit in unpublished work.
Many analytic properties of the usual ($q$\nobreakdash-)derivative continue to hold for the $F$-derivative: for instance, we show (see \cref{cor: F-log}) that there is an explicit power series $F\log(x)$ which recovers the $q$-logarithm when $F$ is the $q$-integer GNS, and which satisfies the property that $\nabla_F(F\log(x)) = 1/x$.

Our main results regarding the $F$-de Rham complex can be summarized as follows:
\begin{ABCthm}[\cref{thm: omnibus-fgl} and \cref{thm: BL-analogue}]
Let $R$ be a torsionfree (say) commutative ring, and let $F$ be a formal group law over $R$.
\begin{enumerate}
    \item Let $R\pw{t}\pdb{x}_F$ denote the ring $R\pw{t}[x, \frac{x^n}{[n]_F!}]_{n\geq 0}$. Then the Poincar\'e lemma holds: the cohomology of the complex $\fdr{\AA^1} \otimes_{R\pw{t}[x]} R\pw{t}\pdb{x}_F$ is concentrated in degree zero, where it is isomorphic to $R\pw{t}$.
    \item The Cartier isomorphism holds: after setting $\pdb{p}_F = 0$, the $i$th cohomology of the complex $\fdr{\AA^1}$ is isomorphic to the $i$th term of a Frobenius twist of $\fdr{\AA^1}$.
    \item There is an analogue of the d\'ecalage isomorphism of \cite{berthelot-ogus, bhatt-scholze} for $\fdr{\AA^1}$.
    \item Using the aforementioned $F$-analogue $F\log(x)$ of the $q$-logarithm, we prove a generalization of the Cartesian square of \cite[Lemma 3.5.18]{apc}.
\end{enumerate}
\end{ABCthm}
Except for the final part, the above result in fact admits a generalization to arbitary GNS (not just ones which arise from formal group laws), but the statement is slightly more complicated; see \cref{sec: s-Cartier}.
As with the combinatorial results above, \cref{thm: omnibus-fgl} is not technically involved; however, it is supposed to serve as a blueprint for a more general program that we outline at the end of \cref{sec: formal group law n-series and the s-derivative}. In particular, we state the (almost certainly false) \cref{conj: polynomial maps}, stating that the assignment $R[x_1, \ldots, x_n] \mapsto (\fdr{\AA^1})^{\otimes_{R\pw{t}} n}$ should extend to a functor from the category of commutative $R$-algebras to the $\infty$-category of $\Eoo$-$R\pw{t}$-algebras. 

\subsection{Table of commonly-used notation}

This article will introduce some notation which will be used heavily throughout. For the reader's convenience, we have summarized the commonly-used ones in the table below.
\vspace{.5cm}
\begin{center}
\begin{tabular}{ c|c|c } 
\hline
Symbol & Definition & Location in text \\
\hhline{=|=|=}
$R[1/s]$ & $R[s(1)^{-1}, s(2)^{-1}, \cdots]$ & \cref{eq: r-localization} \\
\hline
$c_s(n,k)$ & $\frac{s(n) - s(k)}{s(n - k)}$ & \cref{eq: csnk-definition} \\ 
\hline
$(x+y)^n_s$ & Characterized by specific conditions & \cref{def: (x + y)^n_s} \\
\hline
$C_s(n,k)$ & $\frac{s(n + k) - s(n) - s(k)}{s(n) s(k)}$ & \cref{def: capital-CS-ab} \\ 
\hline
$\combbinom{n}{j}{m}{k}_s$ & $\sum_{m < i_1 < i_2 < \cdots < i_j \le m + n} \left( \prod_{\ell = 1}^j c_s(i_\ell, i_\ell - k + j - \ell)\right)$ & \cref{def: s-combbinom} \\
\hline
$\ab{n}_F(t)$ & $\frac{[n]_F(t)}{t}$ for a FGL $F(x,y)$ & \cref{recall: fgl-recall} \\
\hline
$\ell_F(t)$, $\ce_F(t)$ & Logarithm and exponential of a FGL & \cref{recall: fgl-recall} \\
\hline
$F\log(x)$ & $\frac{t}{\ell_F(t)} \log(x)$ & \cref{cor: F-log} \\
\hline
$\GG_m^{\sharp,F}$ & ``$F$-divided power hull'' of zero section of $\GG_m$ & \cref{def: Gm-sharp F} \\
\hline
\end{tabular}
\end{center}
\vspace{.5cm}

\subsection{Acknowledgements}

The first author is grateful to Ben Antieau, Dimitar Jetchev, and (especially) Arpon Raksit for discussions on this topic, as well as Michael Kural for his help in proving the main result of \cref{sec: bhatt-lurie-analogue} in the case of the multiplicative formal group law. The second author is grateful to Pavel Etingof for suggesting that an $s$-analogue of the $q$-binomial theorem might exist.
Finally, both authors are grateful to the directors and organizers of PRIMES-USA for the opportunity to collaborate, as well as for comments on this article!

\newpage
\section{The $s$-Binomial Coefficients}
\label{sec: the s-binomial coefficients}

\subsection{A generalization of binomial coefficients}
\label{sec: a generalization of binomial coefficients}

Recall that the binomial coefficient $\binom{n}{k}$ is defined by \[\binom{n}{k} = \frac{n!}{k! (n - k)!}.\] The $q$-factorial and $q$-binomial coefficients are defined by \[[n]!_q = [1]_q \cdot [2]_q \cdot \cdots \cdot [n]_q, \qquad \binom{n}{k}_q = \frac{[n]!_q}{[k]!_q [n - k]!_q},\] where $[n]_q = (q^n - 1) / (q - 1) \in \Z\pw{q-1}$. The similarity of these two definitions hints that it might be interesting to study a simultaneous generalization of the usual and $q$-binomial coefficients, where the sequence of elements $n\in \Z$ and $[n]_q\in \Z\pw{q-1}$ are replaced by a sequence of elements in a commutative ring satisfying certain conditions.
\begin{definition}
\label{def: s-binomial coefficient}
Let $R$ be a ring, and let $s : \Z_{\ge 0} \to R$ be a function such that $s(0) = 0$ and for all $n > 0$, $s(n)$ is not a zero-divisor. For integers $n \ge k \ge 0$, we define the \textit{$s$-factorial} $n!_s$ by \[0!_s = 1, \quad n!_s = \prod_{k = 1}^n s(k),\] and the \textit{$s$-binomial coefficient} $\binom{n}{k}_s$ by \[\binom{n}{k}_s = \frac{n!_s}{k!_s (n - k)!_s}.\]
\end{definition}
\begin{remark}
In general, this quotient is undefined in $R$; however, it is always defined in the localization
\begin{equation}\label{eq: r-localization}
    R[1/s] := R[s(1)^{-1}, s(2)^{-1}, s(3)^{-1}, \dots].
\end{equation}
We will soon restrict to the case where the $s$-binomial coefficients are elements of $R$.
\end{remark}
The $s$-binomial coefficient need not satisfy any nice properties, since there are no restrictions placed on $s$. Our first observation is the following.
\begin{prop}
\label{thm: s-Pascal identity}
    Let $R$ be a ring and let $s : \Z_{\ge 0} \to R$ be a function that satisfies the following conditions:
    \begin{enumerate}
        \item $s(0) = 0$,
        \item $s(n)$ is not a zero-divisor for any $n > 0$,
        \item $s(n - k) \mid s(n) - s(k)$ for all $n > k > 0$.
    \end{enumerate}
    Then, for all integers $n \ge k \ge 0$, the $s$-binomial coefficient $\binom{n}{k}_s$ is an element of $R$, and the $s$-binomial coefficients satisfy an ``$s$-Pascal identity'': For all $n > k > 0$, \[\binom{n}{k}_s = \binom{n - 1}{k - 1}_s + \frac{s(n) - s(k)}{s(n - k)} \binom{n - 1}{k}_s.\]
\end{prop}
\begin{proof}
Indeed, observe that in the localization $R[1/s]$, we have:
\begin{align*}
    \frac{s(n) - s(k)}{s(n - k)} \binom{n - 1}{k}_s & = \frac{s(n)-s(k)}{s(n-k)} \frac{(n - 1)!_s}{k!_s (n - k-1)!_s}\\
    & = \frac{s(n) \cdot (n - 1)!_s}{k!_s (n - k)!_s} - \frac{s(k) \cdot (n - 1)!_s}{k!_s (n - k)!_s} \\
    & = \frac{n!_s}{k!_s (n - k)!_s} - \frac{(n - 1)!_s}{(k-1)!_s (n - k)!_s} \\
    & = \binom{n}{k}_s - \binom{n-1}{k-1}_s.
\end{align*}
It remains to show that $\binom{n}{k}_s \in R$ for all $n \ge k \ge 0$. We will use induction on $n$.

The base case is clear, since $\binom{0}{0}_s = 1 \in R$. 
For the inductive step, assume that for some fixed $n$ and for all $k$ with $n - 1 \ge k \ge 0$, $\binom{n - 1}{k}_s \in R$. Let $k$ be an integer such that $n \ge k \ge 0$. If $k = 0$, then $\binom{n}{k}_s = 1 \in R$. Otherwise, we can apply the $s$-Pascal identity: \[\binom{n}{k}_s = \binom{n - 1}{k - 1}_s + \frac{s(n) - s(k)}{s(n - k)} \binom{n - 1}{k}_s.\] By the inductive hypothesis, the two $s$-binomial coefficients on the right-hand side are in $R$, and by condition (c) in the theorem statement, $\frac{s(n) - s(k)}{s(n - k)} \in R$. Therefore, $\binom{n}{k}_s \in R$. This completes the induction proof.
\end{proof}
Motivated by \cref{thm: s-Pascal identity}, we are led to the following:
\begin{definition}
\label{def: generalized n-series}
    Let $R$ be a ring. A \textit{generalized $n$-series (GNS) over $R$} is a function $s : \Z_{\ge 0} \to R$ such that the following conditions are true:
    \begin{enumerate}
        \item $s(0) = 0$,
        \item $s(n)$ is not a zero-divisor for any $n > 0$,
        \item $s(n - k) \mid s(n) - s(k)$ for all $n > k > 0$.
    \end{enumerate}
If $s$ is a generalized $n$-series, we will define
\begin{equation}\label{eq: csnk-definition}
    c_s(n,k) := \frac{s(n)-s(k)}{s(n-k)}.
\end{equation}
\end{definition}
\begin{example}[Integers]\label{ex: integer}
The inclusion $s: \Z_{\geq 0} \to \Z$ is clearly a GNS over $\Z$.
\end{example}
\begin{example}[$q$-integers]\label{ex: q-integer}
Consider the function $s: \Z_{\geq 0} \to \Z\pw{q-1}$ given by $s(n) = [n]_q$. This defines a GNS: the first two conditions are satisfied, since $[0]_q = 0$, $[n]_q \ne 0$ for $n > 0$, and $\Z\pw{q-1}$ is an integral domain. For the third condition, note that
\begin{align*}
[n]_q - [k]_q
&= \frac{q^n - 1}{q - 1} - \frac{q^k - 1}{q - 1} = \frac{q^n - q^k}{q - 1} \\
& = q^k \left(\frac{q^{n - k} - 1}{q - 1}\right) = q^k [n - k]_q.
\end{align*}
Therefore, $s$ is a GNS over $\Z\pw{q-1}$, and we can apply \cref{thm: s-Pascal identity} to conclude that $\binom{n}{k}_q \in \Z\pw{q-1}$ for all $n \ge k \ge 0$. The $s$-Pascal identity reduces to the well-known $q$-Pascal identity: \[\binom{n}{k}_q = \binom{n - 1}{k - 1}_q + q^k \binom{n - 1}{k}_q.\]
\end{example}
\begin{remark}\label{rem: arbitrary-integer-binom}
One can extend the definition of the $s$-binomial coefficients to allow arbitrary integers $k$ by defining $\binom{n}{k}_s = 0$ when $k < 0$ or $k > n$. Using this extended definition, the $s$-Pascal identity remains true when $k = 0$: \[\binom{n - 1}{-1}_s + \frac{s(n) - s(0)}{s(n - 0)} \binom{n - 1}{0}_s = 0 + 1 \cdot 1 = 1 = \binom{n}{0}_s.\] This relies on the condition $s(0) = 0$. The fact that Pascal's identity fails for $k = 0$ if $s(0) \ne 0$ is one motivation for including condition (1) in the definition of GNS. 
\end{remark}
\subsection{Number-theoretic properties of generalized $n$-series}
\label{sec: number-theoretical properties of GNS}

In this section, we prove some number-theoretic properties of generalized $n$-series, which will be useful later in this article. The main result of this section is the following:
\begin{theorem}
\label{thm: number-theoretical properties of generalized n-series}
    Let $s$ be a generalized $n$-series over a ring $R$. Then, for all $a, b, n \in \Z_{\ge 0}$,
    \begin{enumerate}
        \item $a \mid b \implies s(a) \mid s(b)$,
        \item $a \equiv b \pmod{n} \implies s(a) \equiv s(b) \pmod{s(n)}$,
        \item the ideals $(s(a), s(b))$ and $(s(\gcd(a, b)))$ are equal.
    \end{enumerate}
    If $s(1)$ is a unit in $R$, then for all $a, n \in \Z_{\ge 0}$,
    \begin{enumerate}[resume]
        \item $a \text{ is a unit in } \Z/n \implies s(a) \text{ is a unit in } R/s(n)$.
    \end{enumerate}
\end{theorem}
\begin{remark}
In the case $R = \Z$, the equivalence of ideals in Theorem \ref{thm: number-theoretical properties of generalized n-series} is equivalent to \[\gcd(s(a), s(b)) = \pm s(\gcd(a, b)).\]
\end{remark}
We will prove \cref{thm: number-theoretical properties of generalized n-series} as a sequence of lemmas. Fix a generalized $n$-series $s$ over a ring $R$.

\begin{lemma}
\label{lem: divisibility of generalized n-series}
Let $a, b \in \Z_{\ge 0}$. Then, $a \mid b$ implies $s(a) \mid s(b)$.
\end{lemma}

\begin{proof}
We will use induction. Base case: $s(a) \mid s(0)$ because $s(0) = 0$. Inductive hypothesis: Let $n \in \Z_{> 0}$, and assume that $s(a) \mid s(a(n-1))$. Then, by the divisibility condition in the definition of generalized $n$-series, \[s(a) \mid s(a(n-1)) = s(an - a) \mid s(an) - s(a),\] so $s(a) \mid s(an)$.
\end{proof}

\begin{lemma}
\label{lem: congruence of generalized n-series}
Let $a, b, n \in \Z_{\ge 0}$. If $a \equiv b \pmod n$ then \[s(a) \equiv s(b) \pmod{s(n)}.\] Another way to state this lemma is that $s$ induces a well-defined function from $\Z/n$ to $R/s(n)$.
\end{lemma}

\begin{proof}
By the definition of congruence, $n \mid a - b$, so $s(n) \mid s(a - b)$ by Lemma \ref{lem: divisibility of generalized n-series}. The definition of generalized $n$-series requires that \[s(a - b) \mid s(a) - s(b),\] so $s(n) \mid s(a) - s(b)$, which means that $s(a) \equiv s(b) \pmod{s(n)}$.
\end{proof}

\begin{lemma}
\label{lem: s of a unit mod n is a unit mod s(n)}
Suppose that $s(1)$ is a unit in $R$. If $a, n \in \Z_{\ge 0}$ such that $a$ is a unit in $\Z/n$, then $s(a)$ is a unit in $R/s(n)$.
\end{lemma}

\begin{proof}
Let $b$ be the multiplicative inverse of $a$ modulo $n$. Then, $ab \equiv 1 \pmod{n}$. By Lemmas \ref{lem: divisibility of generalized n-series} and \ref{lem: congruence of generalized n-series}, \[s(a) \mid s(ab) \equiv s(1) \pmod{s(n)}.\] So in the ring $R/s(n)$, $s(a)$ divides $s(1)$, which is a unit (because it is a unit in $R$). Therefore, $s(a)$ is a unit in $R/s(n)$.
\end{proof}

\begin{lemma}
\label{lem: s preserves gcd}
Let $a, b \in \Z_{\ge 0}$. Then, we have the following equivalence of ideals: \[\big(s(\gcd(a, b))\big) = \big(s(a), s(b)\big).\]
\end{lemma}

\begin{proof}
Let $d = \gcd(a, b)$. By Lemma \ref{lem: divisibility of generalized n-series}, $s(d) \mid s(a)$ and $s(d) \mid s(b)$, so $s(a), s(b) \in \big(s(d)\big)$. This means that \[\big(s(d)\big) \supseteq \big(s(a), s(b)\big).\] For the other direction, we can use B\'ezout's identity to write $d = am + bn$ for some $m, n \in \Z$. Taking this equation modulo $a$ gives $d \equiv bn \pmod{a}$. By Lemma \ref{lem: congruence of generalized n-series}, $s(d) \equiv s(bn) \pmod{s(a)}$. Therefore, \cref{lem: divisibility of generalized n-series} implies that \[s(d) \in s(bn) + \big(s(a)\big) \subseteq \big(s(a), s(b)\big),\]
and hence $\big(s(d)\big) \subseteq \big(s(a), s(b)\big)$. This shows that the two ideals are equal.
\end{proof}

\subsection{The $s$-binomial theorem}
\label{sec: the s-binomial theorem}

\begin{recall}
The binomial theorem and the $q$-binomial theorem are the following two identities:
\begin{align*}
(x + y)^n &= \sum_{k = 0}^n \binom{n}{k} x^{n - k} y^k, \\
(x + y)^n_q &= \sum_{k = 0}^n \binom{n}{k}_q q^{k(k - 1)/2} x^{n - k} y^k,
\end{align*}
where $(x + y)^n_q$ is defined by \[(x + y)^n_q = \prod_{k = 0}^{n - 1} (x + q^k y) = (x + y)(x + qy)\cdots(x + q^{n - 1}y).\] 
Note that the $x + y$ in the parentheses is part of the notation, and cannot be treated as a sum; see \cite{quantum_calculus} for this notation.
\end{recall}
\begin{remark}
For readers who are familiar with the $q$-Pochhammer symbol, $(x + y)^n_q = x^n (-y/x; q)_n$, and $(a; q)_n = (1 + (-a))^n_q$.
\end{remark}

We will now state and prove an analogue of the binomial theorem for the $s$-binomial coefficients.
We begin by defining an analogue of the symbol $(x+y)^n_q$. To motivate the definition, recall that the $q$-analogue $(x + y)^n_q$ is the unique polynomial in $\Z\pw{q-1}[x, y]$ such that the following properties hold:
\begin{itemize}
    \item The $q$-derivative with respect to $x$ of $(x + y)^n_q$ is $[n]_q (x + y)^{n - 1}_q$. This is analogous to the fact that the classical derivative of $(x+y)^n$ with respect to $x$ is $n(x + y)^{n - 1}$.
    \item $(x + y)^0_q = 1$.
    \item If $y = -x$, then $(x + y)^n_q$ is $0$ for all $n > 0$.
\end{itemize}
To define an $s$-analogue $(x + y)^n_q$ in a similar way, we need an $s$-derivative; we will greatly expand on this notion in \cref{sec: formal group laws}.

\begin{definition}
\label{def: s-derivative}
    Let $s: \Z_{\geq 0} \to R$ be a GNS. The \textit{$s$-derivative} is the $R$-linear map $\nabla_s : R[x] \to R[x]$ given on monomials by $\nabla_s(x^n) = s(n) x^{n - 1}$.
\end{definition}
\begin{remark}
When $n = 0$, we have $\nabla_s(x^0) = s(0) x^{-1}$. This is not defined in $R[x]$ unless $s(0) = 0$, which is always true when $s$ is a GNS. Continuing \cref{rem: arbitrary-integer-binom}, this observation is another reason for requiring $s(0) = 0$ in the definition of GNS.
\end{remark}

We can now define $(x + y)^n_s$:
\begin{definition}
\label{def: (x + y)^n_s}
    Let $s$ be a GNS over $R$, so that $R[1/s] = R[s(1)^{-1}, s(2)^{-1}, \dots]$. Define $(x + y)^n_s$ for $n \in \Z_{\ge 0}$ to be the unique polynomial in $R[1/s][x, y]$ such that the following three conditions hold:
    \begin{enumerate}
        \item $(x + y)^0_s = 1$,
        \item $(x + (-x))^n_s = 0$ for all $n > 0$,
        \item $\nabla_{s, x} (x + y)^n_s = s(n) (x + y)^{n - 1}_s$.
    \end{enumerate}
    Here, $\nabla_{s, x}: R[1/s][x,y] \to R[1/s][x,y]$ is the operator given by the ``$s$-derivative with respect to $x$'': it is simply the $R[1/s][y]$-linear extension of the $s$-derivative $\nabla_s : R[x] \to R[x]$ to $R[1/s][x, y]$.
\end{definition}

\begin{lemma}\label{lem: sum-well-defined}
The symbol $(x + y)^n_s$ in \cref{def: (x + y)^n_s} is well-defined: it exists and is unique. Moreover, $(x + y)^n_s$ is a homogeneous polynomial of degree $n$.
\end{lemma}

\begin{proof}
    We will use induction on $n$. For the base case $n = 0$, observe that $(x + y)^0_s = 1$ by condition (1). 
    
    For the inductive step, fix $n > 0$, and suppose that for all $k < n$, $(x + y)^k_s$ is well-defined and homogeneous of degree $k$.
    We can $s$-antidifferentiate $s(n) (x + y)^{n - 1}_s$ using the $R[1/s][y]$-linear operator $I_{s, x} : R[1/s][x, y] \to R[1/s][x, y]$ defined on monomials by $I_{s, x}(x^k) = s(k + 1)^{-1} x^{k + 1}$. By definition, this operator produces polynomials with no term of $x$-degree $0$. Although the $s$-antiderivative \[f(x, y) = I_{s, x}(s(n) (x + y)^{n - 1}_s)\] is homogeneous of degree $n$ (since the operator $I_{s, x}$ increases $x$-degree by $1$) and satisfies condition (3), it might not equal $(x + y)^n_s$ because it does not have to satisfy condition (2). Since $f(x, y)$ is homogeneous of degree $n$, $f(x, -x)$ is a scalar multiple of $x^n$, say $ax^n$. Then, the polynomial \[g(x, y) = f(x, y) - a(-y)^n\] satisfies \[\nabla_{s, x} g(x, y) = \nabla_{s, x} f(x, y) = s(n) (x + y)^{n - 1}_s\] and \[g(x, -x) = f(x, -x) - a(-(-x))^n = ax^n - ax^n = 0.\] It follows that $(x + y)^n_s$ exists, and one possible value for it is $g(x, y)$, which is homogeneous of degree $n$.

    It remains to show that $(x + y)^n_s$ is unique. We know that any polynomial $h(x, y)$ that satisfies the conditions of $(x + y)^n_s$ must match $g(x, y)$ in every term with positive $x$-degree, because their $s$-derivatives with respect to $x$ are both $s(n) (x + y)^{n - 1}_s$. Therefore, $h(x, y) - g(x, y)$ is a scalar multiple of $y^n$, say $by^n$. Setting $y = -x$ gives $b(-x)^n = h(x, -x) - g(x, -x)$, which is $0$ by condition (2), so $b = 0$. Therefore, $h(x, y) = g(x, y)$. This proves that $(x + y)^n_s$ is unique and is equal to $g(x, y)$.
\end{proof}

Recall from \cref{sec: the s-binomial coefficients} that we originally defined the $s$-binomial coefficients as elements of the ring $R[1/s] = R[s(1)^{-1}, s(2)^{-1}, \dots]$, and later proved (using the $s$-Pascal identity) that if $s$ is a GNS, then all the $s$-binomial coefficients are elements of $R$. We will do something similar for $(x + y)^n_s$ below, and we will use the $s$-binomial theorem as a lemma in the proof that $(x + y)^n_s \in R[x, y]$. Here is the $s$-binomial theorem:

\begin{theorem}[$s$-binomial theorem]\label{thm: s-binomial theorem}
    Let $s$ be a GNS over $R$. Then, as elements of $R[1/s][x,y]$, we have:
    \[(x + y)^n_s = \sum_{k = 0}^n \binom{n}{k}_s x^{n - k} y^k (0 + 1)^k_s.\]
\end{theorem}

\begin{proof}
    We will use induction on $n$, and the inductive step will mainly consist of applying the $s$-antidifferentiation operator $I_{s, x}$ from the proof of \cref{lem: sum-well-defined} to both sides. For the base case, observe that if $n = 0$, both sides are $1$. 
    
    For the inductive step, fix $n > 0$, and assume that the $s$-binomial theorem is true for $n - 1$: \[(x + y)^{n - 1}_s = \sum_{k = 0}^{n - 1} \binom{n - 1}{k}_s x^{n - k - 1} y^k (0 + 1)^k_s.\] Multiplying both sides by $s(n)$, applying $I_{s, x}$, and using $R[1/s][y]$-linearity gives:
    \begin{align*}
        I_{s, x}(s(n) (x + y)^{n - 1}_s)
        &= s(n) \sum_{k = 0}^{n - 1} \binom{n - 1}{k}_s I_{s, x}(x^{n - k - 1}) y^k (0 + 1)^k_s \\
        &= \sum_{k = 0}^{n - 1} \frac{s(n)}{s(n - k)} \binom{n - 1}{k}_s x^{n - k} y^k (0 + 1)^k_s.
    \end{align*}
    Notice that \[\frac{s(n)}{s(n - k)} \binom{n - 1}{k}_s = \frac{s(n) (n - 1)!_s}{s(n - k) k!_s (n - k - 1)!_s} = \frac{n!_s}{k!_s (n - k)!_s} = \binom{n}{k}_s.\] This implies that \[I_{s, x}(s(n) (x + y)^{n - 1}_s) = \sum_{k = 0}^{n - 1} \binom{n}{k}_s x^{n - k} y^k (0 + 1)^k_s.\] The right-hand side almost looks like the right-hand side of the $s$-binomial theorem that we are trying to prove, but the upper limit of the summation is $n - 1$ instead of $n$. To fix this, add $y^n (0 + 1)^n_s$ to both sides, giving \[I_{s, x}(s(n) (x + y)^{n - 1}_s) + y^n (0 + 1)^n_s = \sum_{k = 0}^n \binom{n}{k}_s x^{n - k} y^k (0 + 1)^k_s.\] We just have to show that the left-hand side is equal to $(x + y)^n_s$.

    Let $g(x, y)$ be the left-hand side. The $s$-derivative with respect to $x$ of $g(x, y)$ is $s(n) (x + y)^{n - 1}_s$, because $I_{s, x}$ is an $s$-antiderivative operator (a right inverse of $\nabla_{s, x}$) and $y^n (0 + 1)^n_s$ is constant with respect to $x$. This is equal to the $s$-derivative of $(x + y)^n_s$, so $g(x, y)$ matches $(x + y)^n_s$ in all terms with positive $x$-degree. Both $g(x, y)$ and $(x + y)^n_s$ are homogeneous of degree $n$, so the only terms with $x$-degree $0$ are the $y^n$ terms. The coefficient of $y^n$ in $g(x, y)$ is $(0 + 1)^n_s$ because $I_{s, x}$ never produces terms with $x$-degree $0$. The $y^n$ term of $(x + y)^n_s$ is $(0 + y)^n_s$, which is $y^n (0 + 1)^n_s$ by homogeneity, so the coefficient of $y^n$ in $(x + y)^n_s$ is also $(0 + 1)^n_s$. Therefore, $g(x, y) = (x + y)^n_s$, so \[(x + y)^n_s = \sum_{k = 0}^n \binom{n}{k}_s x^{n - k} y^k (0 + 1)^k_s.\] This is what we needed to show for the inductive step.
\end{proof}

To complete this subsection, we will show that the coefficients of $(x + y)^n_s$ are elements of $R$ for all GNS $s$ over $R$ and all $n \in \Z_{\ge 0}$. This means that the $s$-binomial theorem is really an equivalence of polynomials in $R[x, y]$, and can be stated without using the larger ring $R[1/s][x, y]$.

\begin{prop}
    Let $s$ be a GNS over a ring $R$. For all nonnegative integers $n$, we have $(x + y)^n_s \in R[x, y]$.
\end{prop}

\begin{proof}
    It suffices to show that $(0 + 1)^n_s \in R$ for all nonnegative integers $n$, because using the $s$-binomial theorem, we could conclude that \[(x + y)^n_s = \sum_{k = 0}^n \binom{n}{k}_s x^{n - k} y^k (0 + 1)^k_s \in R[x, y].\] Setting $x = -1$ and $y = 1$ in the $s$-binomial theorem, we get \[0 = ((-1) + 1)^n_s = \sum_{k = 0}^n \binom{n}{k}_s (-1)^{n - k} (0 + 1)^k_s,\] so
    \begin{equation}\label{eq: recurrence (0+1)^n_s}
        (0 + 1)^n_s = \sum_{k = 0}^{n - 1} \binom{n}{k}_s (-1)^{n - k - 1} (0 + 1)^k_s.
    \end{equation}
    This is a recurrence relation for $(0 + 1)^n_s$. 
    
    To prove that $(0+1)^n_s\in R$, we will use induction on $n$. For the base case, note that $(0 + 1)^0_s = 1$ which is an element of $R$. For the inductive step, note that if $(0 + 1)^k_s \in R$ for all $k < n$, then by the recurrence relation \cref{eq: recurrence (0+1)^n_s} and the fact that the $s$-binomial coefficients are in $R$ (\cref{thm: s-Pascal identity}), \[(0 + 1)^n_s = \sum_{k = 0}^{n - 1} \binom{n}{k}_s (-1)^{n - k - 1} (0 + 1)^k_s \in R.\] This completes the induction.
\end{proof}
\begin{remark}
Using the recurrence relation \cref{eq: recurrence (0+1)^n_s}, it can be shown that \[\frac{(-1)^n (0 + 1)^n_s}{n!_s} = \sum_{\ell = 1}^n \left(\sum_{k_1 + \cdots + k_\ell = n} \left(\prod_{j = 1}^\ell \frac{(-1)}{k_\ell!_s} \right)\right),\] where the inner sum is over all ordered $\ell$-tuples $(k_1, k_2, \dots, k_\ell)$ of positive integers that sum to $n$. Combining the inner and outer sums, the right-hand side can be viewed as a sum over compositions (ordered partitions) of $n$. Isolating $(0 + 1)^n_s$ gives 
\[(0 + 1)^n_s = \sum_{\pi \text{ composition of } n} (-1)^{n - |\pi|} \binom{n}{\pi_1, \pi_2, \dots, \pi_{|\pi|}}_s.\] The summand is an $s$-multinomial coefficient, defined by \[\binom{n}{k_1, k_2, \dots, k_n}_s = \frac{n!_s}{k_1!_s \cdot k_2!_s \cdot \cdots \cdot k_n!_s},\] and $|\pi|$ denotes the length of $\pi$.
\end{remark}
\subsection{The $s$-Lucas theorem}
\label{sec: the s-Lucas theorem}

\begin{recall}
Lucas's theorem says that for all primes $p$ and all integers $n_1, n_0, k_1, k_0 \in \Z_{\ge 0}$ such that $n_0, k_0 < p$, \[\binom{n_1 p + n_0}{k_1 p + k_0} \equiv \binom{n_1}{k_1} \binom{n_0}{k_0} \pmod{p}.\] There is a $q$-analogue of this identity, known as the $q$-Lucas theorem: \[\binom{n_1 p + n_0}{k_1 p + k_0}_q \equiv \binom{n_1}{k_1} \binom{n_0}{k_0}_q \pmod{[p]_q}.\]
This identity is also true if $p$ is composite, as long as we replace the modulus $[p]_q$ with the cyclotomic polynomial $\Phi_p(q)$. Notice that the first binomial coefficient on the right-hand side of the $q$-Lucas theorem is \textit{not} a $q$-binomial coefficient.
\end{recall}

Here is an $s$-analogue of Lucas's theorem:

\begin{theorem}[$s$-Lucas theorem]
\label{thm: s-Lucas theorem with extra condition}
Let $s$ be a GNS over $R$ such that $s(1) = 1$ and
\begin{equation}\label{eq: s-lucas-criterion}
    s(a + b) \equiv s(a) + s(b) \pmod{s(a) s(b)}
\end{equation}
for all $a, b \in \Z_{> 0}$. Then, for any prime $p$ and any nonnegative integers $n_1, n_0, k_1, k_0$ such that $n_0, k_0 < p$, we have \[\binom{n_1 p + n_0}{k_1 p + k_0}_s \equiv \binom{n_1}{k_1} \binom{n_0}{k_0}_s \pmod{s(p)}.\]
\end{theorem}
\begin{remark}\label{rmk: cyclotomic polynomial}
Define $\Phi_n(s) = \prod_{d \mid n} s(d)^{\mu(n/d)}$, where $\mu$ denotes the M\"obius function. Note that M\"obius inversion implies the identity $s(n) = \prod_{d \mid n} \Phi_n(s)$. One can prove a ``composite version'' of the $s$-Lucas theorem where $s(p)$ is replaced by $\Phi_n(s)$, but the proof is more complicated.
\end{remark}

For the rest of this section, we fix the GNS $s$ over $R$. To simplify the statement of the $s$-Lucas theorem, we will use the following:

\begin{notation}\label{def: capital-CS-ab}
For integers $a, b \in \Z_{> 0}$, we define \[C_s(a, b) = \frac{s(a + b) - s(a) - s(b)}{s(a) s(b)} \in R[1/s].\]
The extra condition \cref{eq: s-lucas-criterion} in \cref{thm: s-Lucas theorem with extra condition} is therefore equivalent to $C_s(a, b)$ being defined in $R$ for all $a, b \in \Z_{> 0}$.
\end{notation}

\begin{lemma}
\label{lem: C defined implies c(mn, mk) is 1 mod s(m)}
Suppose that $C_s(a, b)$ is defined in $R$ for all $a, b \in \Z_{> 0}$. Then, for all integers $m > 0$ and $n > k \ge 0$, \[c_s(mn, mk) \equiv 1 \pmod{s(m)}.\]
\end{lemma}

\begin{proof}
By definition, \[c_s(mn, mk) - 1 = \frac{s(mn) - s(mk) - s(mn - mk)}{s(mn - mk)} = s(mk) C_s(mk, mn - mk).\] By Lemma \ref{lem: divisibility of generalized n-series}, $s(m) \mid s(mk)$, so the right-hand side is divisible by $s(m)$. Therefore, $c_s(mn, mk) - 1 \equiv 0 \pmod{s(m)}$.
\end{proof}

\begin{remark}\label{rem: s_m rescaling}
A consequence of this lemma is the interesting fact that if we define the ``rescaled'' GNS $s_m(n) = s(mn)$ for each positive integer $m$, then Pascal's identity for the $s_m$-binomial coefficients is the same as the usual Pascal's identity when we reduce modulo $s(m)$. This implies the following congruence:
\end{remark}

\begin{lemma}
\label{lem: s_m binom congruent to usual binom mod s(m)}
Suppose that $C_s(a, b)$ is defined for all $a, b \in \Z_{> 0}$. For all integers $m > 0$ and $0 \le k \le n$, \[\binom{n}{k}_{s_m} \equiv \binom{n}{k} \pmod{s(m)}.\]
\end{lemma}

\begin{proof}
We will use induction on $n$. For the base case, observe that if $n = 0$ then $k = 0$, so both sides are $1$. For the inductive step, assume that the desired claim is true when $n$ is replaced by $n - 1$. Then:
\begin{align*}
\binom{n}{k}_{s_m}
&= \binom{n - 1}{k - 1}_{s_m} + c_{s_m}(n, k) \binom{n - 1}{k}_{s_m} &\text{by the $s$-Pascal identity (\cref{thm: s-Pascal identity})}, \\
&= \binom{n - 1}{k - 1}_{s_m} + c_s(mn, mk) \binom{n - 1}{k}_{s_m} &\text{by definition of $c_{s_m}$}, \\
&\equiv \binom{n - 1}{k - 1}_{s_m} + \binom{n - 1}{k}_{s_m} &\text{by Lemma \ref{lem: C defined implies c(mn, mk) is 1 mod s(m)}}, \\
&\equiv \binom{n - 1}{k - 1} + \binom{n - 1}{k} \qquad &\text{by inductive hypothesis}, \\
&= \binom{n}{k} \pmod{s(m)} &\text{by Pascal's identity}.
\end{align*}
Therefore, by induction, this is true for all $m \in \Z_{> 0}$.
\end{proof}

\cref{thm: s-Lucas theorem with extra condition} is a consequence of a slight variant.
\begin{prop}[Another $s$-Lucas theorem]
\label{thm: s-Lucas theorem}
Let $s$ be a GNS such that $s(1) = 1$. Let $p$ be prime and let $n_1, n_0, k_1, k_0 \in \Z_{\ge 0}$ such that $n_0, k_0 < m$. Then, \[\binom{n_1 p + n_0}{k_1 p + k_0}_s \equiv \binom{n_1}{k_1}_{s_p} \binom{n_0}{k_0}_s \pmod{s(p)},\] where $s_p$ is the rescaled GNS from \cref{rem: s_m rescaling}.
\end{prop}
\begin{proof}[Proof of \cref{thm: s-Lucas theorem with extra condition}]
This is an immediate consequence of \cref{thm: s-Lucas theorem} and \cref{lem: s_m binom congruent to usual binom mod s(m)}.
\end{proof}
\begin{remark}
\cref{thm: s-Lucas theorem} is just \cref{thm: s-Lucas theorem with extra condition} but with $\binom{n_1}{k_1}$ replaced by $\binom{n_1}{k_1}_{s_p}$, and no requirement that $C_s(a, b)$ be an element of $R$.
\end{remark}
Before we prove \cref{thm: s-Lucas theorem} in full generality, we will prove the special case where $n_0 = k_0 = 0$. Then, we will use this case as a lemma in the proof of the general result.

\begin{proof}[Proof of Theorem \ref{thm: s-Lucas theorem} in the case $n_0 = k_0 = 0$]
Writing out the definition of the $s$-binomial coefficient $\binom{pn}{pk}_s$, we get \[\binom{pn}{pk}_s = \frac{s(pn) s(pn - 1) \cdots s(pn - pk + 1)}{s(pk) s(pk - 1) \cdots s(1)}.\] By \cref{lem: s of a unit mod n is a unit mod s(n)}, $s(pn - j)$ is a unit modulo $s(p)$ for all $j$ not divisible by $p$. Together with \cref{lem: congruence of generalized n-series}, this implies that we can cancel $s(pn - j)$ with $s(pk - j)$ (since they are congruent units modulo $s(p)$). After all of this cancellation, we are left with \[\binom{pn}{pk}_s = \frac{s(pn) s(p(n - 1)) \cdots s(p(n - k + 1))}{s(pk) s(p(k-1)) \cdots s(p)},\] which is just $\binom{n}{k}_{s_p}$. This means that \[\binom{pn}{pk}_s \equiv \binom{n}{k}_{s_p} \pmod{s(p)}.\] Combining this with the congruence of $\binom{n}{k}_{s_p}$ and $\binom{n}{k}$, we get \[\binom{pn}{pk}_s \equiv \binom{n}{k} \pmod{s(p)},\] which is the special case of Theorem \ref{thm: s-Lucas theorem} where $n_0 = k_0 = 0$.
\end{proof}

We will now extend this to any value of $n_0$ in the valid range $0 \le n_0 < p$.

\begin{proof}[Proof of Theorem \ref{thm: s-Lucas theorem} in the case $k_0 = 0$]
We proved \cref{thm: s-Lucas theorem} above for $n_0 = 0$, so let $0 < n_0 < p$. Notice that the definition of $s$-binomial coefficients implies that \[s(n - k) \binom{n}{k}_s = \frac{n!_s}{k!_s (n - k - 1)!_s} = s(n) \binom{n - 1}{k}_s.\] 
Therefore,
\[s(p(n_1 - k) + n_0) \binom{pn_1 + n_0}{pk}_s = s(pn_1 + n_0) \binom{pn_1 + n_0 - 1}{pk}_s.\] Since $n_0$ is a unit modulo $p$, we see that $s(p(n_1 - k) + n_0)$ is a unit modulo $s(p)$. But $s(p(n_1 - k) + n_0)$ is congruent to $s(pn_1 + n_0)$ by Lemma \ref{lem: congruence of generalized n-series}. Since both are units, this implies that
\[\binom{pn_1 + n_0}{pk}_s \equiv \binom{pn_1 + n_0 - 1}{pk}_s \pmod{s(p)}.\] A simple induction proof gives \[\binom{pn_1 + n_0}{pk}_s \equiv \binom{pn_1}{pk}_s \pmod{s(p)}.\] It follows that \[\binom{pn_1 + n_0}{pk}_s \equiv \binom{pn_1}{pk}_s \equiv \binom{n_1}{k} = \binom{n_1}{k} \binom{n_0}{0}_s \pmod{s(p)},\] which is exactly the $s$-Lucas theorem where $k_0 = 0$.
\end{proof}

Finally, we will extend this by induction to any value of $k_0$ between $0$ and $p - 1$.

\begin{proof}[Proof of Theorem \ref{thm: s-Lucas theorem} in full generality]
We will use induction on $k_0$. We already proved the base case $k_0 = 0$ above. Suppose we have shown that the induction hypothesis \[\binom{pn_1 + n_0}{pk_1 + k_0 - 1}_s \equiv \binom{n_1}{k_1} \binom{n_0}{k_0 - 1}_s \pmod{s(p)}\] is true for some $k_0 - 1$ between $0$ and $p - 2$ (or $0 < k_0 < p$). We will show that it is true with $k_0 - 1$ replaced by $k_0$. Notice that for all integers $1 \le k \le n$, \[s(k) \binom{n}{k}_s = \frac{n!_s}{(k - 1)!_s (n - k)!_s} = s(n - k + 1) \binom{n}{k - 1}_s,\] so \[s(pk_1 + k_0) \binom{pn_1 + n_0}{pk_1 + k_0}_s = s(p(n_1 - k_1) + n_0 - k_0 + 1) \binom{pn_1 + n_0}{pk_1 + k_0 - 1}_s.\] Reducing modulo $s(p)$ gives \[s(k_0) \binom{pn_1 + n_0}{pk_1 + k_0}_s \equiv s(n_0 - k_0 + 1) \binom{pn_1 + n_0}{pk_1 + k_0 - 1} \pmod{s(p)}.\] So
\begin{align*}
s(k_0) \binom{pn_1 + n_0}{pk_1 + k_0}_s
&\equiv s(n_0 - k_0 + 1) \binom{pn_1 + n_0}{pk_1 + k_0 - 1} \\
&\equiv \binom{n_1}{k_1} s(n_0 - k_0 + 1) \binom{n_0}{k_0 - 1}_s \\
&= \binom{n_1}{k_1} s(k_0) \binom{n_0}{k_0}_s \pmod{s(p)}.
\end{align*}
Since $0 < k_0 < p$, $s(k_0)$ is a unit modulo $s(p)$, so we can cancel it from both sides, and we get the congruence we wanted. Therefore, the induction proof is complete.
\end{proof}

When the ring $R$ is $\Z$, there is a version of the $s$-Lucas theorem which allows $p$ to be any natural number, not just a prime; see \cref{thm: composite s-Lucas theorem}.

\subsection{The $s$-Vandermonde identity}
\label{sec: the s-Vandermonde identity}

\begin{recall}
The classical Vandermonde identity states that
$$\binom{m + n}{k} = \sum_j \binom{m}{k - j} \binom{n}{j};$$
this admits a $q$-analogue, known as the the $q$-Vandermonde identity: \[\binom{m + n}{k}_q = \sum_j \binom{m}{k - j}_q \binom{n}{j}_q q^{j(m-k+j)}.\]
\end{recall}

To state an $s$-analogue of the Vandermonde identity, we need a definition.
\begin{definition}
\label{def: s-combbinom}
For $m, n, k, j \in \Z$ with $0 \le j \le n$ and $0 \le k - j \le m$, define \[\combbinom{n}{j}{m}{k}_s = \sum_{\substack{I \subseteq (m, m + n] \cap \Z \\ |I| = j}} \left( \prod_{\ell = 1}^j c_s(i_\ell, n - (k - j + \ell))\right),\] 
where $I = \{i_1 < \cdots < i_j\}$.
We also set $\combbinom{n}{j}{m}{k}_s = 0$ for $j < 0$ or $j > n$.
\end{definition}
\begin{remark}
For $s(n) = n$, the product inside the sum in \cref{def: s-combbinom} is $1$, so the sum is just the number of subsets of $(m, m + n] \cap \Z$ of size $j$, which is $\binom{n}{j}$. In that sense, $\combbinom{n}{j}{m}{k}_s$ is an $s$-analogue of the subset-counting definition of $\binom{n}{j}$ which depends on two extra parameters $m$ and $k$. The other type of $s$-binomial coefficient $\binom{n}{j}_s$ is an $s$-analogue of the algebraic definition of $\binom{n}{j}$. 
\end{remark}

\begin{theorem}[$s$-Vandermonde identity]
\label{thm: s-Vandermonde identity}
For all $m, n, k \in \Z_{\ge 0}$ with $k \le m + n$, \[\binom{m + n}{k}_s = \sum_j \binom{m}{k - j}_s \combbinom{n}{j}{m}{k}_s,\] where the sum is taken over all $j \in \Z$ such that $0 \le j \le n$ and $0 \le k - j \le m$.
\end{theorem}
\begin{question}
As explained in \cite{mathoverflow-q-vandermonde}, the $q$-Vandermonde identity for $\binom{m+n}{k}_q$ can be understood as arising via a motivic cellular decomposition of the Grassmannian $\Gr_k(\cc^{m+n})$. Is there a motivic interpretation of \cref{thm: s-Vandermonde identity}? (A similar question can also be asked for the $s$-analogues of the other combinatorial identities proved elsewhere in this article.)
\end{question}

The proof of \cref{thm: s-Vandermonde identity} requires a preliminary lemma, which can be viewed as an analogue of Pascal's identity for $\combbinom{n}{j}{m}{k}_s$.
\begin{lemma}
\label{lem: combbinom s-Pascal identity}
For all $m, n, k, j \in \Z$ such that $0 \le k - j \le m$, \[\combbinom{n}{j}{m}{k}_s = \combbinom{n - 1}{j}{m}{k}_s + c_s(m + n, m+n - k) \combbinom{n - 1}{j - 1}{m}{k - 1}_s.\]
Notice that this includes the cases $j < 0$ and $j > n$.
\end{lemma}
\begin{proof}
We will split up the sum in the definition of $\combbinom{n}{j}{m}{k}_s$ into two sums, based on whether $I$ contains $m + n$ or not. If $m + n \notin I$, then $I$ ranges over all $j$-element subsets of $(m, m + n - 1] \cap \Z$. If $m + n \in I$, then we remove it, and we get a $(j - 1)$-element subset of $(m, m + n - 1] \cap \Z$. Then, we have to pull out the $\ell = j$ term of the product, since $i_j = m + n$. So we have
\begin{align*}
\combbinom{n}{j}{m}{k}_s
&= \sum_{\substack{I \subseteq (m, m + n] \cap \Z \\ |I| = j}} \left( \prod_{\ell = 1}^j c_s(i_\ell, i_\ell - (k - j + \ell))\right) \\
&= \sum_{\substack{I \subseteq (m, m + n] \cap \Z \\ |I| = j \\ m + n \in I}} \left( \prod_{\ell = 1}^j c_s(i_\ell, i_\ell - (k - j + \ell))\right) + \sum_{\substack{I \subseteq (m, m + n - 1] \cap \Z \\ |I| = j}} \left( \prod_{\ell = 1}^j c_s(i_\ell, i_\ell - (k - j + \ell))\right) \\
&= \sum_{\substack{I \subseteq (m, m + n - 1] \cap \Z \\ |I| = j - 1}} c_s(m + n, m+n - (k - j + j)) \left( \prod_{\ell = 1}^{j - 1} c_s(i_\ell, i_\ell - (k - j + \ell))\right) + \combbinom{n - 1}{j}{m}{k}_s \\
&= c_s(m + n, m+n - k) \combbinom{n - 1}{j - 1}{m}{k - 1}_s + \combbinom{n - 1}{j}{m}{k}_s.
\end{align*}
Rearranging this completes the proof.
\end{proof}

\begin{proof}[Proof of \cref{thm: s-Vandermonde identity}]
We will prove this by induction on $n$. For the base case, observe that if $n = 0$, then all $j$ in the range of summation satisfy $0 \le j \le 0$, so either the sum is empty (if $k - 0 > m$) or its only term is $j = 0$ (if $k - 0 \le m$). The conditions of the theorem force $k \le m + n = m$, so the sum has exactly one term: \[\sum_j \binom{m}{k - j}_s \combbinom{n}{j}{m}{k}_s = \binom{m}{k}_s \combbinom{0}{0}{m}{k}_s.\] We want to show that this is $\binom{m}{k}_s$. The definition of $\combbinom{0}{0}{m}{k}_s$ gives the rather degenerate identity \[\combbinom{0}{0}{m}{k}_s = \sum_{\substack{I \subseteq (m, m] \cap \Z \\ |I| = 0}} \left( \prod_{\ell = 1}^0 c_s(i_\ell, i_\ell - (k + \ell)) \right) = \sum_{I \subseteq \emptyset} 1 = 1,\] which completes the base case.

For the inductive step, suppose that the theorem is true with $n$ replaced by $n - 1$. We will first apply a version of the $s$-Pascal identity to $\binom{m + n}{k}_s$ that has been ``flipped'' using the identity $\binom{n}{k}_s = \binom{n}{n - k}_s$:
\begin{align*}
\binom{m + n}{k}_s
&= \binom{m + n}{m + n - k}_s \\
&= \binom{m + n - 1}{m + n - k - 1}_s + c_s(m + n, m + n - k) \binom{m + n - 1}{m + n - k}_s \\
&= \binom{m + n - 1}{k}_s + c_s(m + n, m+n - k) \binom{m + n - 1}{k - 1}_s.
\end{align*}
The inductive hypothesis gives:
\begin{align*}
& \binom{m + n - 1}{k}_s + c_s(m + n, m+n - k) \binom{m + n - 1}{k - 1}_s \\
&= \sum_j \binom{m}{k - j}_s \combbinom{n - 1}{j}{m}{k}_s + c_s(m+n, m+n - k) \sum_j \binom{m}{k - 1 - j}_s \combbinom{n - 1}{j}{m}{k - 1}_s \\
&= \sum_j \binom{m}{k - j}_s \combbinom{n - 1}{j}{m}{k}_s + c_s(m+n, m+n - k) \sum_j \binom{m}{k - j}_s \combbinom{n - 1}{j - 1}{m}{k - 1}_s \\
&= \sum_j \binom{m}{k - j}_s \left( \combbinom{n - 1}{j}{m}{k}_s + c_s(m+n, m+n - k) \combbinom{n - 1}{j - 1}{m}{k - 1}_s \right).
\end{align*}
By Lemma \ref{lem: combbinom s-Pascal identity}, this is equal to \[\sum_j \binom{m}{k - j}_s \combbinom{n}{j}{m}{k}_s,\] which completes the induction.
\end{proof}

\newpage
\section{Generalized $n$-Series Over $\Z$}
\label{sec: integer generalized n-series}

\subsection{Lexicographically small nonnegative integer generalized $n$-series}
\label{sec: small generalized n-series}

When doing computations with generalized $n$-series, it is useful to have some examples over $\Z$ that are lexicographically small (close to $0$ for small values of $n$). We will restrict to integer GNS that are always nonnegative. We have already seen some examples of relatively small integer generalized $n$-series: $s(n) = n$ and $s(n) = [n]_q$. Before looking at an algorithm for constructing lexicographically small GNS, we need a definition:

\begin{definition}
Let $R$ be a ring, and let $N \in \Z_{\ge 0}$. A \textit{partial generalized $n$-series} is a function $s : \{0, 1, \dots, N\} \to R$ such that $s(0) = 0$, $s(n)$ is not a zero-divisor for $0 < n \le N$, and for all $0 \le k \le n \le N$, $s(n - k) \mid s(n) - s(k)$. An \textit{extension} of a partial GNS $s$ is a GNS which agrees with $s$ on the domain of $s$.
\end{definition}

Given a partial GNS \[s : \{0, 1, \dots, N\} \to \Z_{\ge 0}\] we can use a greedy algorithm to construct the lexicographically smallest nonnegative GNS $\tilde s$ which is an extension of $s$.

\begin{definition}
\label{def: extension of partial generalized n-series}
Suppose we are given a partial GNS $s : \{0, 1, \dots, N\} \to \Z_{\ge 0}$. Define a function $\tilde s : \Z_{\ge 0} \to \Z_{\ge 0}$ in the following way: $\tilde s(n) = s(n)$ for $0 \le n \le N$, and for each $n > N$ we define $\tilde s(n)$ in terms of the previous values of $\tilde s(k)$ to be the smallest positive integer that makes the restriction $\tilde s|_{\{0, 1, \dots, n\}}$ a partial GNS.
\end{definition}

The fact that $\tilde s$ is well-defined is nontrivial. We have to show that at each step of the algorithm, $\tilde s(n)$ exists. We will actually prove a stronger theorem:

\begin{theorem}
\label{thm: recursive construction of generalized n-series}
Consider the following recursive construction of an arbitrary nonnegative GNS over $\Z$: Start with $s(0) = 0$, and for each $N$ starting with $1$ in increasing order, choose $s(N)$ to be an \emph{arbitrary} integer such that $s|_{\{0, 1, \dots, N\}}$ is a partial GNS. No matter what choices are made, it is always possible to continue (there are never any contradictions).
\end{theorem}

\begin{proof}
Suppose we have a partial GNS $s$ with domain $\{0, 1, \dots, N - 1\}$, and we want to choose $s(N)$. We have to show that there exists $s(N)$ such that for all $0 < k < N$, $s(N - k) \mid s(N) - s(k)$. This is equivalent to $s(N) \equiv s(k) \pmod{s(N - k)}$, so we really have a system of linear congruences. This system has a solution by the generalized Chinese Remainder Theorem as long as no two congruences contradict each other. We want to show that if we reduce two of the congruences $s(N) \equiv s(N - k) \pmod{s(k)}$ and $s(N) \equiv s(N - j) \pmod{s(j)}$ modulo $\gcd(s(k), s(j))$, they become the same congruence. That is, we want to show that \[s(N - k) \equiv s(N - j) \pmod{\gcd(s(k), s(j))}.\] By \cref{lem: s preserves gcd}, $\gcd(s(k), s(j)) = \pm s(\gcd(k, j))$. And since \[N - k \equiv N \equiv N - j \pmod{\gcd(k, j)},\] the desired congruence must be true by \cref{lem: congruence of generalized n-series}.
\end{proof}

We have seen that many number-theoretic properties of the positive integers remain true for arbitrary GNS. However, there are some important differences between $s(n) = n$ and other GNS. For example, if $n = p^j m$ with $p\nmid m$, then $n/p^j$ is a unit modulo $p$. This is generally false for other generalized $n$-series; using \cref{thm: recursive construction of generalized n-series}, we can construct a counterexample.
\begin{example}
Let $s$ be an extension of the partial GNS with domain $\{0, 1, \dots, 6\}$ whose values are $0, 1, 2, 3, 10, 11, 12$. One can check manually that this forms a partial GNS. For this GNS, \[\frac{s(6)}{s(2)} = \frac{12}{2} = 6 \equiv 0 \pmod{s(2)}.\]
\end{example}

\subsection{An upper bound on lexicographically small nonnegative integer GNS}
\label{sec: upper bound}

In this section, we will write partial generalized $n$-series as lists of numbers, where the first item in the list is $s(0)$. For example:
\begin{example}
Define a partial generalized $n$-series $0, 1, 3$ via the function $s : \{0, 1, 2\} \to \Z_{\ge 0}$ given by $0 \mapsto 0$, $1 \mapsto 1$, $2 \mapsto 3$. If we apply the algorithm of \cref{def: extension of partial generalized n-series} to the partial generalized $n$-series $0, 1, k$ where $k \in \Z_{> 0}$, we get a generalized $n$-series $0, 1, k, 1, k, \dots$, where the $n$-th term is $0$ if $n = 0$, $1$ if $n$ is odd, and $k$ if $n$ is even and nonzero. If we start with a longer partial generalized $n$-series $s$, say $0, 1, 3, 4$, then the generalized $n$-series $\tilde s$ starts with \[0, 1, 3, 4, 9, 19, 552, 22081,\] and the next few terms are \[219440979, 2669857856653708, 6558922971496604200448626056129.\] This seems to grow very fast when $n$ increases, almost doubling the number of digits each term. As we will see, $\tilde s(n)$ is bounded by
\begin{align*}
\tilde s(n) &< 12^{2^{n-4}}, \qquad n \ge 4, \text{ and} \\
\tilde s(n) &= \Omega_a(a^n), \qquad \forall a \ge 0.
\end{align*}

The upper bound (a special case of \cref{thm: upper bound on lexicographically small GNS}) is proved below, and the lower bound (a special case of \cref{thm: lower bound on GNS over Z}) is proved in the next subsection.
\end{example}
Let us first show that $\tilde s$ is strictly increasing.
\begin{lemma}
\label{lem: strictly increasing s implies strictly increasing s-tilde}
Let $N \ge 3$, and let $s : \{0, 1, \dots, N\} \to \Z_{\ge 0}$ be a strictly increasing partial generalized $n$-series. Then, $\tilde s$ is also strictly increasing.
\end{lemma}
\begin{remark}
The condition $N \ge 3$ is important, because we already saw that if $s = 0, 1, 3$ (so $N = 2$) then $\tilde s = 0, 1, 3, 1, 3, \dots$, which is not strictly increasing.
\end{remark}
\begin{proof}
Notice that $s(2) > s(1) \ge 1$. Let $n > N$. We will show that $\tilde s(n) > \tilde s(n - 1)$. By \cref{lem: congruence of generalized n-series}, $\tilde s(n) \equiv s(1) \pmod{s(n - 1)}$, so either $\tilde s(n) = s(1)$ or $\tilde s(n) > s(n - 1)$. The first case cannot happen because $\tilde s(n) \equiv s(2) \pmod{\tilde s(n - 2)}$, and $\tilde s(n - 2) \ge s(2) > s(1)$. Therefore, $\tilde s(n) > \tilde s(n - 1)$.
\end{proof}

Given the value of $\tilde s(k)$ for all $1 \le k < n$, there is an upper bound on $\tilde s(n)$, as long as $\tilde s$ is strictly increasing:

\begin{lemma}
\label{lem: s-tilde(n) < lcm(previous s-tilde(k)'s)}
Let $s : \{0, 1, \dots, N\} \to \Z$ be a strictly increasing partial generalized $n$-series with $N \ge 3$. Then, for all $n > N$, \[\tilde s(n) < \lcm\{\tilde s(k) \mid 0 < k < n\}.\]
\end{lemma}

\begin{proof}
Recall that the proof of the existence of $\tilde s$ (\cref{thm: recursive construction of generalized n-series}) constructs $\tilde s(n)$ from $\tilde s(k)$ for all $0 < k < n$ using the generalized Chinese Remainder Theorem. The congruences are $x \equiv \tilde s(n - k) \pmod{\tilde s(k)}$ for each $0 < k < n$, so the generalized CRT proves the existence of a unique solution $x$ modulo $L = \lcm\{\tilde s(k) \mid 0 < k < n\}$. Therefore, there exists a solution to the system of congruences with $0 < x \le L$. The smallest positive solution is $\tilde s(n)$ by \cref{def: extension of partial generalized n-series}, so $\tilde s(n) \le L$. We just have to show that $\tilde s(n) \ne L$.

Suppose for contradiction that $\tilde s(n) = L$. Then, by \cref{lem: congruence of generalized n-series}, $L \equiv \tilde s(1) \pmod{\tilde s(n - 1)}$. But by the definition of $L$, $L$ is divisible by $\tilde s(n - 1)$, so $\tilde s(1)$ is also divisible by $\tilde s(n - 1)$. Since $n - 1 > N - 1 > 1$, $\tilde s(1) < \tilde s(n - 1)$ by \cref{lem: strictly increasing s implies strictly increasing s-tilde}, which is a contradiction.
\end{proof}

Using the fact that the LCM of a set is at most its product, we can prove the following upper bound:

\begin{theorem}
\label{thm: upper bound on lexicographically small GNS}
Let $s : \{0, 1, \dots, N\} \to \Z$ be a strictly increasing partial generalized $n$-series with $N \ge 3$. Then, for all $n > N$, \[\tilde s(n) < \Pi^{2^{n - (N + 1)}},\] where $\Pi = \prod_{k = 1}^N s(k)$.
\end{theorem}

\begin{proof}
We will use strong induction. Suppose that $n > N$ and for all $k$ strictly between $N$ and $n$, \[\tilde s(k) < \Pi^{2^{k - (N + 1)}}.\] By \cref{lem: s-tilde(n) < lcm(previous s-tilde(k)'s)} and the inductive hypothesis,
\begin{align*}
\tilde s(n)
&< \lcm\{\tilde s(k) \mid 0 < k < n\} \le \prod_{k = 1}^{n - 1} \tilde s(k) \\
&\le \left(\prod_{k = 1}^N s(k)\right) \left(\prod_{k = N + 1}^{n - 1} \Pi^{2^{k - (N + 1)}}\right) \\
&= \Pi \cdot \prod_{k = 0}^{n - (N + 2)} \Pi^{2^k} = \Pi \cdot \Pi^{\left(\sum_{k = 0}^{n - (N + 2)} 2^k\right)} \\
&= \Pi \cdot \Pi^{2^{n - (N + 1)} - 1} = \Pi^{2^{n - (N + 1)}}.
\end{align*}
This completes the induction.
\end{proof}

\begin{remark}
Another way to think about this theorem is that the recursively defined sequence \[a_n = \begin{cases} s(n) &\text{if } 0 \le n \le N, \\ \prod_{k = 1}^{n - 1} a_k &\text{if } n > N \end{cases}\] is exactly equal to $\Pi^{2^{n - (N + 1)}}$ for all $n > N$, and we know that $\tilde s(n)$ satisfies this recurrence but with the equality replaced by $<$ whenever $n > N$, so it should be true that $\tilde s(n) < a_n$ for all $n > N$.
\end{remark}

\subsection{A lower bound on strictly increasing integer GNS}
\label{sec: lower bound}

We proved an upper bound on the lexicographically smallest extension of the partial GNS $0, 1, 3, 4$ in the previous section, but we also stated a lower bound. The purpose of this subsection is to prove this lower bound by proving a more general result.

\begin{theorem}
\label{thm: lower bound on GNS over Z}
Let $s(n)$ be a strictly increasing generalized $n$-series over $\Z$ that is not a scalar multiple of $n \mapsto [n]_q$ for any $q \in \Z_{> 0}$. Then, $s(n) = \Omega_a(a^n)$ for all $a \ge 0$.
\end{theorem}

In this theorem statement, we use the convention that $[n]_1 = n$. To prove \cref{thm: lower bound on GNS over Z}, we will use the following lemma:

\begin{lemma}
\label{lem: lower bound on GNS over Z lemma}
Let $s$ satisfy the conditions of \cref{thm: lower bound on GNS over Z}, and fix a nonnegative integer $a$. Then, for all sufficiently large $n$, \[s(n + 1) \ne a s(n) + s(1).\]
\end{lemma}

\begin{proof}
Let $k \in \Z_{\ge 0}$ such that $s(k + 1) \ne a s(k) + s(1)$. Such a $k$ must exist, because otherwise $s(n + 1) = a s(n) + s(1)$ for all $n \in \Z_{\ge 0}$. This would imply by induction that $s(n) = s(1) [n]_q$ with $q = a$ for all $n$, which contradicts our assumption about $s$. Next, choose an integer $N$ large enough so that \[s(N - k) > \max\{s(k + 1), a s(k) + s(1)\}.\] This is always possible, because $s : \Z_{\ge 0} \to \Z$ is strictly increasing and therefore unbounded.

To prove the lemma, we will show that for all $n \ge N$, we have $s(n + 1) \ne a s(n) + s(1)$. Suppose for contradiction that there exists $n \ge N$ with \[s(n + 1) = a s(n) + s(1).\] Taking this equation modulo $s(n - k)$ and using \cref{lem: congruence of generalized n-series} gives \[s(k + 1) \equiv a s(k) + s(1) \pmod{s(n - k)}.\] But $n \ge N$, so \[s(n - k) \ge s(N - k) > \max\{s(k + 1), a s(k) + s(1)\},\] so the modulus is greater than both sides of the congruence, which means that it is an equality. This contradicts the fact that $s(k + 1) \ne a s(k) + s(1)$.
\end{proof}

\begin{proof}[Proof of \cref{thm: lower bound on GNS over Z}]
We want to show that $s(n) = \Omega_a(a^n)$ for all $a \ge 0$. It suffices to prove this for $a \in \Z_{\ge 0}$. For each integer $b$ with $0 < b < a$, apply \cref{lem: lower bound on GNS over Z lemma} with $a$ replaced by $b$. This gives an integer $N_b$ such that for all $n \ge N_b$, \[s(n + 1) \ne b s(n) + s(1).\] Let \[N = \max\{N_b \mid 0 < b < a\}.\] We claim that for all $n \ge N$, we have $s(n + 1) \ge a s(n) + s(1)$.

To see this, let $n \ge N$. By \cref{lem: congruence of generalized n-series}, \[s(n + 1) \equiv s(1) \pmod{s(n)},\] so $s(n + 1) = b s(n) + s(1)$ for some $b \in \Z$. Since $s$ is strictly increasing, $b$ must be positive. We want to show that $b \ge a$. If $b < a$, then we defined $N_b$ above, and $n \ge N \ge N_b$. So $s(n + 1) \ne b s(n) + s(1)$, which is a contradiction. Therefore, $b \ge a$, so \[s(n + 1) \ge a s(n) + s(1).\]

To complete the proof, notice that the claim implies (by induction) that for all $n \ge N$, $s(n) \ge a^{n - N} s(N)$, which means that $s(n) = \Omega_a(a^n)$ as $n \to \infty$.
\end{proof}

\subsection{A more general $s$-Lucas theorem}
\label{sec: a more general s-lucas theorem}

The $s$-Lucas theorem (\cref{thm: s-Lucas theorem}) is a congruence modulo $s(p)$, where $s$ is a GNS and $p$ is prime. The $q$-Lucas theorem has a more general form which allows $p$ to be composite: \[\binom{n_1 p + n_0}{k_1 p + k_0}_q \equiv \binom{n_1}{k_1} \binom{n_0}{k_0}_q \pmod{\Phi_p(q)}.\] In this subsection, we will state and prove an $s$-analogue of the more general $q$-Lucas theorem in the case $n_0 = k_0 = 0$.

Recall the $s$-analogue $\Phi_n(s)$ of the cyclotomic polynomial from \cref{rmk: cyclotomic polynomial}.

\begin{theorem}
\label{thm: Phi_n(s) is an integer for GNS over Z}
    If $s$ is a GNS over $\Z$, then $\Phi_n(s) \in \Z$ for all $n > 0$.
\end{theorem}

We will prove that $\Phi_n(s) \in \Z$ by showing that \[\Phi_n(s) = \frac{s(n)}{\lcm\{s(n/p) \mid p \text{ prime factor of } n\}}.\] Here is a lemma:

\begin{lemma}
\label{lem: formula for lcm in terms of gcd}
    Let $S$ be an arbitrary multiset of positive integers. Then, \[\lcm(S) = \prod_{\substack{\text{\emph{multiset }} A \subseteq S \\ A \ne \emptyset}} \gcd(A)^{(-1)^{|A| - 1}}.\]
\end{lemma}

This lemma can be proved using the tools of elementary number theory. Notice that when $s$ is a two-element multiset $\{a, b\}$, this formula reduces to $\lcm(a, b) = ab / \gcd(a, b)$.

\begin{proof}[Proof of \cref{lem: formula for lcm in terms of gcd}]
We can decompose the right-hand side into its prime-power factors. Let $p$ be a prime, and let $T$ be the multiset $\{v_p(a) \mid a \in S\}$, where $v_p(a)$ is the exponent of $p$ in the prime factorization of $a$. The exponent of $p$ in the right-hand side of the equality we are trying to prove is \[\sum_{\substack{\text{multiset } B \subseteq T \\ B \ne \emptyset}} (-1)^{|B| - 1} \min(B),\] because a GCD of powers of $p$ is $p$ to the power of the minimum exponent. We want to show that this is equal to $v_p(\lcm(S))$, which is $\max(T)$. Write $T = \{a_1, a_2, \dots, a_k\}$, where $k = |T|$ and $a_1 \le a_2 \le \cdots \le a_n$. We will count the number of times each $a_j$ is counted in the sum. Since $\max(T) = a_n$, we have to show that $a_n$ is counted once and $a_j$ is counted $0$ times for all $1 \le j < n$.

Let $1 \le j \le n$ and let $1 \le k \le n$. We want to count how many subsets $B \subseteq T$ have minimum $j$ and cardinality $n$. Such a subset must include $a_j$, but can have any combination of $k - 1$ elements $a_\ell$ with $\ell > j$. There are $\binom{n - j}{k - 1}$ choices for these remaining elements. Therefore, the number of times $a_j$ is counted in the sum is \[\sum_{k = 1}^n (-1)^{k - 1} \binom{n - j}{k - 1} = \sum_{k = 0}^{n - 1} (-1)^k \binom{n - j}{k}.\] Since $\binom{n - j}{k}$ is zero for $k > n - j$, this is just the alternating sum of row $n - j$ of Pascal's triangle, which is $0$ if $n - j > 0$ and $1$ if $n - j = 0$. Therefore, $a_j$ is counted $0$ times for $j < n$ and $1$ time for $j = n$ in the sum \[\sum_{\substack{B \subseteq P \\ B \ne \emptyset}} (-1)^{|B| - 1} \min(B),\] so the sum is equal to $a_n$. This means that the exponents of $p$ in the left-hand and right-hand sides of the equality in the theorem statement are both $a_n$. Combining this fact for each prime $p$ proves the lemma.
\end{proof}

\begin{proof}[Proof of \cref{thm: Phi_n(s) is an integer for GNS over Z}]
Let $P$ be the set of prime factors of $n$, and let \[r(n) = \frac{s(n)}{\lcm\{s(n/p) \mid p \in P\}}.\] Notice that $r(n) \in \Z$ because \cref{lem: divisibility of generalized n-series} implies that $s(n/p) \mid s(n)$ for all prime factors $p$ of $n$. We want to show that $\Phi_n(s) = r(n)$.

By the definition of $r(n)$, \[\frac{s(n)}{r(n)} = \lcm\left\{s\left(\frac{n}{p}\right) \bigmid p \in P\right\}.\] By \cref{lem: formula for lcm in terms of gcd} with $S = \{s(n/p) \mid p \in P\}$, this is equal to  \[\prod_{\substack{A \subseteq P \\ A \ne \emptyset}} \gcd\left\{s\left(\frac{n}{p}\right) \bigmid p \in A\right\}^{(-1)^{|A| - 1}}.\] \cref{lem: s preserves gcd} and the fact that $s$ is nonnegative imply that $s$ preserves GCDs, so \[\gcd\left\{s\left(\frac{n}{p}\right) \bigmid p \in A\right\} = s\left(\gcd\left\{\frac{n}{p} \bigmid p \in A\right\}\right) = s\left(\frac{n}{\lcm(A)}\right).\] This implies that \[\frac{s(n)}{r(n)} = \prod_{\substack{A \subseteq P \\ A \ne \emptyset}} s\left(\frac{n}{\lcm(A)}\right)^{(-1)^{|A| - 1}}.\] Dividing both sides by $s(n)$ and taking the reciprocal of both sides gives \[r(n) = \prod_{A \subseteq P} s\left(\frac{n}{\lcm(A)}\right)^{(-1)^{|A|}}.\] (On the right-hand side, we removed the condition $A \ne \emptyset$ and multiplied the exponent by $-1$.)

We have to count how many ways each divisor $d$ of $n$ can be written as $\lcm(A)$ for $A \subseteq P$. Since the elements of $A$ are distinct primes, $\lcm(A)$ is the product of the elements of $A$. Therefore, $d = \lcm(A)$ is square-free, and given $d$ there is a unique choice of $A$ (the set of prime factors of $d$). So we get \[r(d) = \prod_{\substack{d \mid n \\ d \text{ square-free}}} s\left(\frac{n}{d}\right)^{\mu(d)},\] using the fact that $\mu(d)$ is $(-1)^{|A|}$ if $A$ is the set of prime factors of a square-free number $d$. Since $\mu(d) = 0$ for $d$ not square-free, we can remove the restriction that $d$ is square-free: \[r(n) = \prod_{d \mid n} s\left(\frac{n}{d}\right)^{\mu(d)} = \prod_{d \mid n} s(d)^{\mu(n/d)} = \Phi_n(s).\] This is what we wanted to show.
\end{proof}

Before proving the $s$-Lucas theorem for composite $p$, we will prove a lemma about $\Phi_n(s)$.

\begin{lemma}
\label{lem: lemma about Phi_n(s)}
Let $s$ be a nonnegative integer GNS, let $m \in \Z_{> 0}$, and let $a, b \in \Z_{\ge 0}$ such that \[a \equiv b \not\equiv 0 \pmod{m}.\] Define \[d = \gcd(a, m) = \gcd(b, m).\] Then, \[\frac{s(a)}{s(d)} \equiv \frac{s(b)}{s(d)} \pmod{\Phi_m(s)}.\] Additionally, both sides of this congruence are units modulo $\Phi_m(s)$. Equivalently, the function $\Z_{\ge 0} \to \Z_{\ge 0}$ defined by $a \mapsto s(a) / s(\gcd(a, m))$ induces a well-defined function \[(\Z/m) \setminus \{0\} \to (\Z/\Phi_m(s))^\times.\]
\end{lemma}

\begin{proof}
By \cref{lem: congruence of generalized n-series}, \[s(a) \equiv s(b) \pmod{s(m)}.\] 
Since $\gcd(s(m), s(d)) = s(d)$, this implies that
\[\frac{s(a)}{s(d)} \equiv \frac{s(b)}{s(d)} \quad\left(\!\!\!\!\!\!\mod{\frac{s(m)}{s(d)}}\right).\] Notice that
\begin{align*}
\gcd\left\{\frac{s(m)}{s(d)} \bigmid d \text{ proper divisor of } m\right\}
&= \frac{s(m)}{\lcm\{s(d) \mid d \text{ proper divisor of } m\}} \\
&= \frac{s(m)}{\lcm\{s(m/p) \mid p \text{ prime factor of } m\}} \\
&= \Phi_m(s).
\end{align*}
In the second step of this chain of equalities, we used the fact that for every proper divisor $d$ of $m$ there exists a prime factor $p$ of $m$ such that $d \mid m/p$. We also used \cref{lem: divisibility of generalized n-series}. We can conclude that $\Phi_m(s) \mid s(m)/s(d)$ for all proper divisors $d$ of $m$. In particular, for $d = \gcd(a, m) = \gcd(b, m)$, we have \[\frac{s(a)}{s(d)} \equiv \frac{s(b)}{s(d)} \pmod{\Phi_m(s)}.\]

To complete the proof, we just have to show that $s(a) / s(d)$ is a unit modulo $\Phi_m(s)$. This is true because $\Phi_m(s) \mid s(m) / s(d)$ implies that \[\gcd\left(\Phi_m(s), \frac{s(a)}{s(d)}\right) \mid \gcd\left(\frac{s(m)}{s(d)}, \frac{s(a)}{s(d)}\right) = \frac{\gcd(s(m), s(a))}{s(d)} = \frac{s(d)}{s(d)} = 1. \qedhere\]
\end{proof}

\begin{theorem}[``Composite version'' of the $s$-Lucas theorem]
\label{thm: composite s-Lucas theorem}
    Let $m \in \Z_{> 0}$ and let $n, k \in \Z_{\ge 0}$. Then, \[\binom{nm}{km}_s \equiv \binom{n}{k}_{s_m} \pmod{\Phi_m(s)}.\]
\end{theorem}

\begin{proof}
    By definition, \[\binom{nm}{km}_s = \frac{s(nm) s(nm - 1) \cdots s(nm - km + 1)}{s(km) s(km - 1) \cdots s(1)} = \frac{\prod_{j = 1}^{km} s(nm - km + j)}{\prod_{j = 1}^{km} s(j)}.\] Our goal will be to cancel each $s(nm - km + j)$ with the corresponding $s(j)$ modulo $\Phi_m(s)$. The problem is that in general, these are not units modulo $\Phi_m(s)$. To fix this, we will divide $s(nm - km + j)$ and $s(j)$ by $s(\gcd(j, m))$: \[\binom{nm}{km}_s = \frac{\prod_{j = 1}^{km} s(nm - km + j) / s(\gcd(j, m))}{\prod_{j = 1}^{km} s(j) / s(\gcd(j, m))}.\] Since $j \equiv nm - km + j \pmod{m}$, we can apply \cref{lem: lemma about Phi_n(s)} for all $j$ not divisible by $m$, and we get that \[\frac{s(nm - km + j)}{s(\gcd(j, m))} \quad \text{and} \quad \frac{s(j)}{s(\gcd(j, m))}\] are congruent units modulo $\Phi_m(s)$.

    Therefore, we can cancel $s(nm - km + j) / s(\gcd(j, m))$ with $s(j) / s(\gcd(j, m))$ modulo $\Phi_m(s)$ for every $j$ not divisible by $m$. We get
    \begin{align*}
        \binom{nm}{km}_s
        &= \frac{\prod_{j = 1}^{km} s(nm - km + j) / s(\gcd(j, m))}{\prod_{j = 1}^{km} s(j) / s(\gcd(j, m))} \\
        &\equiv \frac{\prod_{\ell = 1}^k s((n - k + \ell) m) / s(m)}{\prod_{\ell = 1}^k s(\ell m) / s(m)} \\
        &= \frac{\prod_{\ell = 1}^k s((n - k + \ell) m)}{\prod_{\ell = 1}^k s(\ell m)} \\
        &= \binom{n}{k}_{s_m} \pmod{\Phi_m(s)}.
    \end{align*}
    We changed product indices from $j$ to $\ell = j/m$, because the terms with $m \nmid j$ were canceled in the second step.
\end{proof}
\begin{remark}
When $m$ is prime, \cref{thm: composite s-Lucas theorem} reduces to the case $n_0 = k_0 = 0$ of the $s$-Lucas theorem (\cref{thm: s-Lucas theorem}). 
\end{remark}

\newpage
\section{Generalized $n$-Series and de Rham Complexes}
\label{sec: formal group laws}

\subsection{Basic properties of the $s$-de Rham complex}
\label{sec: the s-Leibniz rule}

In this subsection, we will define the ``$s$-de Rham complex'' using the $s$-derivative of \cref{def: s-derivative}. Recall that this is the $R$-linear map $\nabla_s : R[x] \to R[x]$ given on monomials by $\nabla_s(x^n) = s(n) x^{n - 1}$. Write $\AA^1 = \spec R[x]$ to denote the affine line over $R$.
\begin{definition}\label{def: s-dR}
The $s$-de Rham complex for $R[x]$ is the $2$-term complex
$$\sdr{\AA^1} := (R[x] \xar{\nabla_s} R[x] dx).$$
Here, the square indicates the dependence of $\sdr{\AA^1}$ on the choice of coordinate $x$.
\end{definition}
\begin{remark}
It is easy to generalize the $s$-de Rham complex to several variables (e.g., by defining $s\Omega_{\AA^n}$ to be $\sdr{\AA^1}^{\otimes_R n}$). Since proving multivariable analogues of the results below is straightforward, we will only study the case of a single variable.
\end{remark}
\begin{example}
Let $s: \Z_{\geq 0} \to \Z\pw{q-1}$ denote the $q$-integer GNS from \cref{ex: q-integer}. Then \cref{def: s-dR} is precisely the $q$-de Rham complex of \cite{scholze-q-def}.
\end{example}
The $s$-Pascal identity of \cref{thm: s-Pascal identity} can be restated in terms of the $s$-derivative:
\begin{lemma}
There is an equality of operators:
$$\left[\frac{\nabla_s^k}{k!_s}, x\right] = \frac{\nabla_s^{k - 1}}{(k - 1)!_s}.$$
\end{lemma}
\begin{proof}
Let $n > k \ge 0$ be integers. Since $\nabla_s(x^n) = s(n) x^{n - 1}$, we have \[\nabla_s^k(x^n) = s(n) s(n - 1) \cdots s(n - k + 1) x^{n - k} = \frac{n!_s}{(n - k)!_s} x^{n - k}.\] This implies that \[\frac{\nabla_s^k}{k!_s} x^n = \binom{n}{k}_s x^{n - k}.\] The $s$-Pascal identity can therefore be stated as: \[\frac{\nabla_s^k}{k!_s} x^n = \frac{\nabla_s^{k - 1}}{(k - 1)!_s} x^{n - 1} + x \frac{\nabla_s^k}{k!_s} x^{n - 1}.\] Rearranging gives \[\frac{\nabla_s^k}{k!_s} x^n - x \frac{\nabla_s^k}{k!_s} x^{n - 1} = \frac{\nabla_s^{k - 1}}{(k - 1)!_s} x^{n - 1}.\] Recognizing the left side as a commutator of operators, this can be written as \[\left[\frac{\nabla_s^k}{k!_s}, x\right] x^{n - 1} = \frac{\nabla_s^{k - 1}}{(k - 1)!_s} x^{n - 1}.\] 
This implies the desired equality of operators.
\end{proof}
\begin{remark}
One can similarly restate the $s$-Lucas theorem (\cref{thm: s-Lucas theorem}) via the $s$-derivative: namely, if the hypotheses of \cref{thm: s-Lucas theorem} are satisfied, then for any prime $p$ and any nonnegative integers $n_1, n_0, k_1, k_0$ such that $n_0, k_0 < p$, we have
$$\frac{\nabla_s^{k_1 p + k_0}}{(k_1 p + k_0)!_s}(x^{n_1 p + n_0}) \equiv \frac{\partial_{x^p}^{k_1}}{k_1!}(x^{n_1 p}) \frac{\nabla_s^{k_0}}{k_0!_s}(x^{n_0}) \pmod{s(p)}.$$
\end{remark}

\begin{prop}[$s$-product rule]
    \label{thm: s-product rule}
    For a GNS $s$ over $R$, define an $R$-bilinear operator $\star_s : R[x] \otimes_R R[x] \to R[x]$ on monomials by \[x^a \star_s x^b = c_s(a + b + 1, a) x^{a + b}\] for all $a, b \in \Z_{\ge 0}$. Then, for all $f, g \in R[x]$, \[\nabla_s(fg) = \nabla_s(f) g + f \star_s \nabla_s(g).\]
\end{prop}

\begin{proof}
    Since $\nabla_s$ is $R$-linear and $\star_s$ is $R$-bilinear, the theorem follows from the case where $f$ and $g$ are monomials $x^a$ and $x^b$. In this case,
    \begin{align*}
        \nabla_s(x^a) x^b + x^a \star_s \nabla_s(x^b)
        &= s(a) x^{a - 1} x^b + x^a \star_s s(b) x^{b - 1} \\
        &= s(a) x^{a + b - 1} + s(b) c_s(a + b, a) x^{a + b - 1} \\
        &= \left(s(a) + s(b) \cdot \frac{s(a + b) - s(a)}{s(b)}\right) x^{a + b - 1} \\
        &= (s(a) + s(a + b) - s(a)) x^{a + b - 1} \\
        &= s(a + b) x^{a + b - 1} \\
        &= \nabla_s(x^{a + b}). \qedhere
    \end{align*}
\end{proof}
\begin{example}
Let $s: \Z_{\geq 0} \to \Z\pw{q-1}$ denote the $q$-integer GNS from \cref{ex: q-integer}. Then 
$$c_s(a+b+1,a) = \frac{[a+b+1]_q - [a]_q}{[b+1]_q} = \frac{q^{a+b+1} - q^a}{q^{b+1}-1} = q^a,$$
so that $x^a \star_s x^b = q^a x^{a+b} = (qx)^a x^b$. In particular, 
$$f(x) \star_s \nabla_s(g(x)) = f(qx) \nabla_q(g(x)),$$
so that \cref{thm: s-product rule} reduces to the usual $q$-Leibniz rule.
\end{example}
\begin{example}
One can check that the function $s: \Z_{\geq 0} \to \Z\pw{q-1}$ given by
$$s(n) = \frac{1}{q-1} \frac{q^n - (2-q)^n}{q^n + (2-q)^n} \in \Z\pw{q-1}$$
defines a GNS; see \cref{ex: fgl-examples}. It follows that
$$c_s(a+b+1,a) = 2q^a(2-q)^a \frac{q^{b+1} + (2-q)^{b+1}}{(q^{a+b+1} + (2-q)^{a+b+1}) (q^a + (2-q)^a)}.$$
In particular, unlike for the $q$-integer GNS, there is no simple expression for $f(x) \star_s g(x)$.
\end{example}
\begin{corollary}\label{cor: dga}
Fix a GNS $s$ over $R$. Then the complex $\sdr{\AA^1}$ naturally admits the structure of a (noncommutative) differential graded $R$-algebra.
\end{corollary}
\begin{proof}
Define a left and right $\sdr{\AA^1}^0 = R[x]$-module structure on $\sdr{\AA^1}^1$ as follows: the right module structure is the obvious one, and the left module structure is given by $g(x) \cdot f(x) dx = g \star_s f(x) dx$. Then the $s$-Leibniz rule of \cref{thm: s-product rule} produces a $\sdr{\AA^1}^0$-bimodule structure on $\sdr{\AA^1}^1$ such that the $s$-derivative satisfies the Leibniz rule; this is precisely the structure of a differential graded $R$-algebra.
\end{proof}
\subsection{$s$-analogues of the Poincar\'e lemma and Cartier isomorphism}
\label{sec: s-Cartier}

The Poincar\'e lemma says that over a field $k$ of characteristic zero, the cohomology of the de Rham complex is concentrated in degree zero (where it is isomorphic to $k$). A version of this statement is also true over $\Z$. Namely, if $\Z\pdb{x} = \Z[x, \frac{x^n}{n!}]_{n\geq 0}$ denotes the divided power envelope of $\Z[x]$, then the cohomology of the complex $\Omega^\bull_{\Z[x]/\Z} \otimes_{\Z[x]} \Z\pdb{x}$ is concentrated in degree zero (where it is isomorphic to $\Z$). This admits a straightforward generalization to the $s$-de Rham complex:
\begin{prop}[$s$-Poincar\'e lemma]\label{prop: s-poincare}
Let $s$ be a GNS over $R$, and let $R\pdb{x}_s$ denote the ring $R[x, \frac{x^n}{n!_s}]_{n\geq 0}$. Then the cohomology of the complex $\sdr{\AA^1} \otimes_{R[x]} R\pdb{x}_s$ is concentrated in degree zero, where it is isomorphic to $R$.
\end{prop}
\begin{proof}
We need to show that the $R$-linear map
$$R\pdb{x}_s \xar{\nabla_s} R\pdb{x}_s dx$$
is surjective, and has kernel $R$. Surjectivity follows from the observation that $\nabla_s \frac{x^n}{n!_s} = \frac{x^{n-1}}{(n-1)!_s}$; this also implies that the kernel of $\nabla_s$ is precisely the $R$-submodule of $R\pdb{x}_s$ generated by the constants.
\end{proof}
\begin{remark}
Note that the ring $R\pdb{x}_s$ is nonzero, since the elements $n!_s\in R$ are not zero-divisors; in fact, $R\pdb{x}_s$ is a subring of $R[1/s][x]$.
\end{remark}
The $s$-de Rham complex also satisfies an analogue of the Cartier isomorphism. 
\begin{recall}
The Cartier isomorphism says that if $A$ is a smooth $\FF_p$-algebra, there is a canonical isomorphism $\Omega^i_{A/\FF_p} \cong \H^i(\Omega^\bull_{A/\FF_p})$. If $\varphi$ denotes the Frobenius on $R$, this isomorphism is roughly given by ``$\frac{\varphi}{p^i}$''. When $A = \FF_p[x]$, one can interpret the Cartier isomorphism as giving a canonical isomorphism
$$\Omega^i_{\Z[x^p]/\Z} \otimes_\Z \FF_p \cong \H^i(\Omega^\bull_{\Z[x]/\Z} \otimes_\Z \FF_p)$$
sending $d(x^p) \mapsto [x^{p-1} dx]$.
In \cite[Proposition 3.4]{scholze-q-def}, Scholze proves a $q$-analogue of the Cartier isomorphism. Let $\Z[\zeta_p]$ denote the quotient $\Z\pw{q-1}/[p]_q$, so that $\zeta_p$ denotes a primitive $p$th root of unity. Then there is a canonical isomorphism
$$\Omega^i_{\Z[x^p]/\Z} \otimes_\Z \Z[\zeta_p] \cong \H^i(q\Omega^\bull_{\Z[x]/\Z} \otimes_{\Z\pw{q-1}} \Z[\zeta_p]).$$
\end{recall}
Both of these results admit an $s$-analogue.
\begin{prop}\label{prop: s-cartier}
Let $s$ be a GNS over $R$ such that $s(1)$ is a unit in $R/s(p)$. Then there is a canonical isomorphism
$$\sdr{\AA^1}^i \otimes_R R/s(p) \cong \H^i(\sdr{\AA^1} \otimes_R R/s(p))$$
sending $x^n\mapsto x^{np}$ in degree zero and $x^n dx \mapsto [x^{np} x^{p-1} dx]$.
\end{prop}
\begin{proof}
Let us first compute $\H^0(\sdr{\AA^1} \otimes_R R/s(p))$, i.e., the kernel of $\nabla_s$. Observe that if $a\in R$, then $ax^n \mapsto as(n) x^{n-1} dx$. If $p \mid n$, then $s(p) \mid s(n)$, so that $\nabla_s(ax^n) = 0\in R[x]/s(p)$. If $p\nmid n$, it follows from \cref{lem: s of a unit mod n is a unit mod s(n)} that $\nabla_s(ax^n) = 0\in R[x]/s(p)$ if and only if $s(p) \mid a$. This implies that $\H^0(\sdr{\AA^1} \otimes_R R/s(p)) \cong R[x^p]/s(p)$.

To calculate $\H^1(\sdr{\AA^1} \otimes_R R/s(p))$, i.e., the cokernel of $\nabla_s$, we need to determine the image of $\nabla_s$. If $ax^n dx$ is in the image of $\nabla_s$ for some $a\in R$, then there must be some $b\in R$ such that $b s(n+1) = a$. If $p \mid n+1$, it follows that $a = 0\in R/s(p)$; if $p\nmid n+1$, then $s(n+1)$ is a unit (by the preceding discussion), so that $b = \frac{a}{s(n+1)}$. It follows that the image of $\nabla_s$ is precisely $\bigoplus_{p\nmid n+1} R/s(p)\{x^n dx\}$, so that 
$$\coker(\nabla_s) = \bigoplus_{n\geq 1} R/s(p)\{x^{np-1} dx\} \cong \bigoplus_{n\geq 1} R/s(p)\{x^{p(n-1)} \cdot x^{p-1} dx\}.$$
This implies that $\H^1(\sdr{\AA^1} \otimes_R R/s(p)) \cong R[x^p]/s(p) d(x^p)$, as desired.
\end{proof}
In fact, the classical and $q$-Cartier isomorphisms are special cases of a more general result due to Berthelot and Ogus \cite{berthelot-ogus}. Let us now review this statement; we will then state and prove the analogue for generalized $n$-series.
\begin{recall}
Let $R$ be a ring, and let $f\in A$ be a non-zero-divisor. If $M^\bull$ is a cochain complex of $R$-modules which is termwise $f$-torsionfree, the \textit{d\'ecalage} $\eta_f M^\bull$ is the subcomplex of $M^\bull[1/f]$ defined via
$$(\eta_f M)^i = \{x\in f^i M^i \mid dx\in f^{i+1} M^{i+1}\}.$$
See \cite[Section 6]{bms-i} and \cite{berthelot-ogus}. One basic property of the d\'ecalage construction is the following. Let $\H^\bull(M/f)$ denote the complex whose underlying graded abelian group is $\bigoplus_{i\in \Z} \H^i(M/f)$, and where the differential is given by the $f$-Bockstein $\beta: \H^i(M/f) \to \H^{i+1}(M/f)$. Then there is a natural isomorphism of complexes
\begin{equation}\label{eq: coh-decalage}
    \eta_f(M)/f \xar{\sim} \H^\bull(M/f).
\end{equation}
We will only need the case when $M^\bull$ is termwise $f$-torsionfree; but let us mention that $\eta_f$ preserves quasi-isomorphisms, and one can extend $\eta_f$ to a (non-exact) functor $L\eta_f: D(R) \to D(R)$ on the the derived category of $R$.

Let $k$ be a perfect field of characteristic $p>0$.
A very special case of a result of Berthelot and Ogus (in \cite{berthelot-ogus}) says that if $W(k)$ is the ring of Witt vectors of $k$ and $A$ is a $W(k)$-algebra, there is a Frobenius\footnote{The existence of a Frobenius on $\Omega_A$ is not obvious in general, and depends on the existence of crystalline cohomology.} semilinear isomorphism $\Omega_A \xar{\sim} L\eta_p \Omega_A$. Suppose for simplicity that $\Omega_A$ is $p$-torsionfree; applying \cref{eq: coh-decalage} with $f = p$ then defines an isomorphism of complexes
\begin{equation}\label{eq: decalage-dR}
    \eta_p(\Omega_A)/p \xar{\sim} \H^\bull(\Omega_A/p).
\end{equation}
Note that if we write $A_0 = A/p$, then $\Omega_A/p \cong \Omega_{A_0/k}$. The Frobenius semilinear isomorphism $\Omega_A \xar{\sim} L\eta_p \Omega_A$ gives a Frobenius semilinear equivalence $\Omega_{A_0/k} \xar{\sim} \eta_p(\Omega_{A_0})/p$. Comparing the left and right-hand sides of \cref{eq: decalage-dR} recovers the Cartier isomorphism for $A_0$.

A similar result was proved in \cite[Theorem 1.16(4)]{bhatt-scholze} for the $q$-de Rham complex: namely, if $A$ is a smooth $\Z_p[\zeta_p]$-algebra, then there is a Frobenius semilinear equivalence $q\Omega_A \xar{\sim} L\eta_{[p]_q} q\Omega_A$. As above, using \cref{eq: coh-decalage} with $f = [p]_q$ recovers the $q$-analogue of the Cartier isomorphism.
\end{recall}
As one might expect, there is a d\'ecalage result for the $s$-de Rham complex, too:
\begin{prop}\label{prop: decalage}
Fix a GNS $s$ over $R$ such that:
\begin{enumerate}
    \item there is a ring endomorphism $\varphi: R \to R$ which sends $s(n) \mapsto \frac{s(np)}{s(p)}$.
    \item $s(1)$ is a unit in $R/s(p)$, and $R$ is $s(p)$-adically complete.
\end{enumerate}
Write $\AA^{1,(p)} = \spec R[x^p]$, and define a map $\Phi: \sdr{\AA^{1,(p)}} \to \sdr{\AA^1}$ via
\begin{align*}
    R[x^p] & \xar{\varphi} R[x], \ \text{in degree }0, \\
    R[x^p] d(x^p) & \xar{\varphi, d(x^p) \mapsto s(p) x^{p-1} dx} R[x] dx, \ \text{in degree }1.
\end{align*}
Then $\Phi$ factors through a quasi-isomorphism $\varphi^\ast \sdr{\AA^{1,(p)}} \xar{\sim} \eta_{s(p)} \sdr{\AA^1}$.
\end{prop}
\begin{proof}
Since $\sdr{\AA^1}$ is $s(p)$-torsionfree, we can directly compute the complex $\eta_{s(p)} \sdr{\AA^1}$. If $f(x) = \sum_{n\geq 0} a_n x^n \in R[x]$, then 
$$\nabla_s f(x) = \sum_{n\geq 0} a_n s(n) x^{n-1}.$$
Since $s(p) \mid s(pj)$ for any $j\geq 0$, we see that $\nabla_s f(x)\in s(p) R[x]$ if and only if $s(p) \mid a_n$ for $p \nmid n$. Therefore,
\begin{align*}
    \eta_{s(p)} \sdr{\AA^1} & = \left(R[x^p] + s(p) R[x] \xar{\nabla_s} s(p) R[x] dx\right).
\end{align*}
The map $\Phi$ clearly factors through the inclusion $\eta_{s(p)} \sdr{\AA^1} \subseteq \sdr{\AA^1}$. It remains to check that the map $\varphi^\ast \sdr{\AA^{1,(p)}} \to \eta_{s(p)} \sdr{\AA^1}$ induces an isomorphism on cohomology. Observe that
\begin{align*}
    \H^0(\sdr{\AA^{1,(p)}}) & \cong R, \\
    \H^1(\sdr{\AA^{1,(p)}}) & \cong \bigoplus_{n\geq 1} R/s(n) \{(x^p)^{n-1} d(x^p)\},
\end{align*}
and similarly
\begin{align*}
    \H^0(\eta_{s(p)} \sdr{\AA^1}) & \cong R, \\
    \H^1(\eta_{s(p)} \sdr{\AA^1}) & \cong \bigoplus_{n\geq 1} (s(p))/(s(np)) \{x^{np-1} dx\} \oplus \bigoplus_{p\nmid m} (s(p))/(s(m) s(p)) \{x^{m-1} dx\}.
\end{align*}
The ring endomorphism $\varphi$ of $R$ sends $s(n) \mapsto \frac{s(np)}{s(p)}$, and $s(p)$ is a non-zero-divisor in $R$, we see that $\varphi$ descends to an isomorphism 
$$\varphi^\ast (R/s(n)) \xar{\sim} R/\tfrac{s(np)}{s(p)} \xar{\sim} (s(p))/(s(np)).$$
If $p\nmid m$, then $s(m)$ is a unit in $R/s(p)$ by \cref{lem: s of a unit mod n is a unit mod s(n)}; since $R$ is $s(p)$-adically complete, this implies that $s(m)$ is a unit in $R$ itself. It follows that $0 \cong R/s(m) \xar{\sim} (s(p))/(s(m) s(p))$.
Putting these together, we see that the map $\varphi^\ast \sdr{\AA^{1,(p)}} \xar{\sim} \eta_{s(p)} \sdr{\AA^1}$ is a quasi-isomorphism.
\end{proof}

In \cref{prop: decalage}, the condition that $R$ be $s(p)$-adically complete is rather powerful. For instance, one of the most basic tools in $p$-adic mathematics is the Legendre formula; this admits an analogue for generalized $n$-series, too.
\begin{prop}[$s$-analogue of Legendre formula]\label{prop: legendre}
Let $s$ be a GNS over $R$ satisfying the hypotheses of \cref{prop: decalage}. For any $n\geq 0$, we have
\begin{align*}
    n!_s & = u \prod_{j\geq 1} \varphi^{j-1}(s(p))^{\lfloor n/p^j\rfloor}
\end{align*}
for some unit $u\in R^\times$.
\end{prop}
\begin{proof}
The argument is a straightforward adaptation of \cite[Lemma 12.6]{bhatt-scholze}. Since $R$ is $s(p)$-adically complete, we know that if $p\nmid m$, then $s(m)$ is a unit in $R$. This implies that
\begin{align*}
    (np)!_s & = u \prod_{i=1}^n s(np) = u \prod_{i=1}^n \left(\frac{s(np)}{s(p)} \cdot s(p)\right)\\
    & = u s(p)^n \prod_{i=1}^n \varphi(s(n)) = u s(p)^n \varphi(n!_s),
\end{align*}
where $u = \prod_{1\leq j \leq np, p\nmid j} s(j)$ is a unit in $R$. Using the above identity to inductively strip powers of $p$ off of $n$ produces the desired claim.
\end{proof}
This implies the following analogue of \cite[Lemma 12.5]{bhatt-scholze}, which gives a criterion for admitting ``$s$-divided powers''.
\begin{corollary}\label{cor: s-divided powers criterion}
Let $s$ be a GNS over $R$ satisfying the hypotheses of \cref{prop: decalage}, and suppose that for any $n\geq 0$, $\varphi(n!_s)$ is a non-zero-divisor in $R/s(p)$. Let $A$ be an $s(p)$-completely flat $R$-algebra equipped with a $R$-linear multiplicative map $\phi: \varphi^\ast A \to A$ and an element $x\in A$ such that $\varphi(x) = x^p$, and $\varphi(x)$ is divisible by $s(p)$.
Then $n!_s \mid x^n$, i.e., $\frac{x^n}{n!_s}$ is well-defined in $A$.
\end{corollary}
\begin{proof}
Because $s(j)$ is a unit in $R$ is $p\nmid j$, it suffices to show: for any $n\geq 0$, if $n!_s \mid x^n$, then $(np)!_s \mid x^{np}$. By (the proof of) \cref{prop: legendre}, it suffices to show that $\varphi(n!_s) s(p)^n \mid x^{np}$. Because $\varphi(n!_s)$ is a non-zero-divisor in $R/s(p)$ (by assumption), it is also a non-zero-divisor in $A/s(p)$ by flatness of $A$. It therefore suffices to show that $\varphi(n!_s)$ and $s(p)^n$ each individually divide $x^{np}$. Since $n!_s \mid x^n$, it is clear that $\varphi(n!_s) \mid \varphi(x^n) = x^{np}$. Since $s(p) \mid x^p$ by assumption, we also see that $s(p)^n \mid x^{np}$, as desired.
\end{proof}
\subsection{Formal group law $n$-series and the $s$-derivative}
\label{sec: formal group law n-series and the s-derivative}

Some of the most important (and accessible) examples of generalized $n$-series come from formal group laws. Throughout this section, we will fix a base commutative ring $R$.
\begin{recall}\label{recall: fgl-recall}
A ($1$-dimensional) \textit{formal group law} over a commutative ring $R$ is a two-variable power series $F(x,y)\in R\pw{x,y}$ such that $F(F(x,y),z) = F(x,F(y,z))$ and $F(x,y) \equiv x + y \pmod{(x,y)^2}$. It will sometimes be convenient to denote $F(x,y)$ by $x +_F y$. A morphism $f: F \to G$ of formal group laws is a power series $f(x)\in R\pw{x}$ such that $f(F(x,y)) = G(f(x), f(y))$; a morphism is called an isomorphism if it admits a compositional inverse.

If $n\geq 0$ is an integer, the $n$-series of $F$ is defined via the formula
$$[n]_F(t) = \overbrace{F(t, F(t, F(t, \dots F(t, t) \dots )))}^n = \overbrace{t +_F t +_F \cdots +_F t}^n.$$
This can be extended to all integers $n\in \Z$ by using the existence of inverses for the formal group law.
The $n$-series $[n]_F(t)$ is an element of $R\pw{t}$, and it is always divisible by $t$.  We will define $\ab{n}_F(t) := [n]_F(t) / t$ (where we agree that $\ab{0}_F(t) = 0$). Sometimes, it will be notationally convenient to simply write these as $[n]_F$ and $\ab{n}_F$ (it being implicit that these are functions of $t$).
If $R$ is an $\FF_p$-algebra, then either $[p]_F(t) = 0$ or $[p]_F(t) = \lambda t^{p^h} + O(t^{p^h + 1})$ for some $h>0$. If $v_j$ denotes the coefficient of $t^{p^j}$ in $[p]_F(t)$, then $F$ is said to be of \textit{height $\geq n$} if $v_j = 0$ for $j<n$; if $v_n$ is a unit, then $F$ is said to be of \textit{height $n$}.

If $\QQ\subseteq R$, then every formal group law $F(x,y)$ is isomorphic to the additive formal group law via the \textit{logarithm}. Let $F_y(x,y) = \partial_y F(x,y)$; then, the logarithm is given by the integral
$$\ell_F(x) := \int^x_0 \frac{dt}{F_y(t,0)}.$$
We will write $\ce_F(x)$ to denote its compositional inverse, so that $F(x,y) = \ce_F(\ell_F(x) + \ell_F(y))$. Observe that $[n]_F(t) = \ce_F(n\ell_F(t))$ for any $n\in \Z$.
\end{recall}
\begin{example}\label{ex: fgl-examples}
Fix a base commutative ring $R$.
The polynomial $F(x,y) = x+y$ is known as the \textit{additive} formal group law, and $\ab{n}_F(t) = n$. The polynomial $F(x,y) = x+y+xy$ is known as the \textit{multiplicative} formal group law, and
$$\ab{n}_F(t) = \frac{(1+t)^n-1}{t} \in R\pw{t}.$$
Note that $\pdb{n}_F = [n]_q$, where we set $q = t+1$. The power series $F(x,y) = \frac{x+y}{1+xy}$ is known as the \textit{hyperbolic} formal group law (since it describes the addition law for $\tanh$), and a simple induction on $n$ shows that
$$\ab{n}_F(t) = \frac{1}{t} \frac{(1+t)^n - (1-t)^n}{(1+t)^n + (1-t)^n} = \frac{1}{q-1} \frac{q^n-(2-q)^n}{q^n+(2-q)^n} \in R\pw{q-1}.$$
\end{example}
\begin{remark}\label{rmk: rescaled-fgl}
Given a formal group law $F(x,y)$ over a (torsionfree, say) commutative ring $R$, one can define a ``rescaled'' formal group law $\tilde{F}(x,y)$ over $R\pw{t}$ which is characterized by the property that $\tilde{F}(x,y) = \frac{1}{t} F(xt, yt)$. Observe that over the special fiber (i.e., $t=0$), $\tilde{F}(x,y)$ degenerates to the additive formal group law.
Over $(R\otimes \QQ)\pw{t}$, the logarithm $\tilde{\ell}_F(x)$ is given by $\frac{1}{t} \ell_F(tx)$; note that this power series does not have polar terms in $t$, since $x \mid \ell_F(x)$ (and hence $tx \mid \ell_F(tx)$).
\end{remark}

We begin by showing that the map $n \mapsto [n]_F(t)$ is a GNS over $R\ps{t}$, as long as $F$ satisfies a mild condition. The existence of the power series $\ell_F$ is the main reason that GNS arising via formal group laws are particularly well-behaved.
\begin{prop}\label{prop: [n]_F(t) GNS}
Let $F$ be a formal group law over a ring $R$, and suppose that $[n]_F(t)\in R\pw{t}$ is not a zero-divisor for any $n > 0$. Define $s : \Z_{\ge 0} \to R\ps{t}$ by $s(n) = [n]_F(t)$. Then, $s$ is a GNS over $R\ps{t}$. Similarly, the function $s_F: \Z_{\geq 0} \to R\pw{t}$ sending $s_F(n) = \ab{n}_F(t)$ is a GNS over $R\pw{t}$.
\end{prop}
\begin{proof}
It is easy to see that if $s$ is a GNS, the same is true of $s_F$. Let us now show that $s$ is a GNS by checking the conditions of \cref{def: generalized n-series}.
Condition (1) is clear, since $[0]_F = 0$. Condition (2) is already assumed in the theorem statement.

For condition (3), let $G(x) \in R\ps{t}\ps{x}$ be the power series $F(t, x)$. Then, \[[n + 1]_F - [k + 1]_F = G([n]_F) - G([k]_F).\] Since $x - y$ divides $x^j - y^j$ for each $j\geq 0$, and $G$ is a power series, $x - y$ also divides $G(x) - G(y)$. In particular, $s(n)-s(k) = [n]_F - [k]_F$ divides $G([n]_F) - G([k]_F)$; but $G([n]_F) - G([k]_F) = [n + 1]_F - [k + 1]_F$ is precisely $s(n + 1) - s(k + 1)$, so $s(n)-s(k)$ divides $s(n+1)-s(k+1)$. Inducting, we conclude that $s(n)-s(k)$ divides $s(n+j) - s(k+j)$ for all $n, k, j \in \Z_{\ge 0}$. For $k = 0$, this gives $s(n) \mid s(n + j) - s(j)$ for all $n, j \in \Z_{\ge 0}$. After relabeling the indices, this becomes $s(n - k) \mid s(n) - s(k)$, which proves condition (3).
\end{proof}
\begin{lemma}
    Let $F$ be a formal group law over a ring $R$. If $n \in \Z_{> 0}$ is not a zero-divisor in $R$ (e.g., $R$ is torsionfree), then $[n]_F$ and $\ab{n}_F$ are not zero-divisors in $R\ps{t}$.
\end{lemma}

\begin{proof}
    Suppose for the sake of contradiction that $[n]_F(t)$ is a zero-divisor for some $n$ (the same argument works for $\ab{n}_F$). Then there exists a power series $f(t) \in R\ps{t}$ such that \[f(t) \cdot [n]_F(t) = 0.\] It follows that the product of the coefficients of the lowest-degree terms of $f(t)$ and $[n]_F(t)$ must be $0$. Since the lowest-degree term of $[n]_F(t)$ is $nt$, this implies that $n$ times the lowest-degree coefficient of $f(t)$ is $0$. In particular, $n$ is a zero-divisor in $R$.
\end{proof}

\begin{definition}
\label{def: f-dR}
Let $F$ be a formal group law over $R$ such that $[n]_F(t)\in R\pw{t}$ is not a zero-divisor for any $n > 0$.\footnote{Many of the results below do not rely on this assumption; but we keep it nonetheless, since \cref{prop: [n]_F(t) GNS} allows to immediately transport many results about generalized $n$-series obtained above. The interested reader should have no trouble removing this condition as necessary in the results below.}
Let $\fdr{\AA^1}$ denote the differential graded $R\pw{t}$-algebra given by the $s$-de Rham complex associated to the GNS $s: \Z_{\geq 0} \to R$ sending $n\mapsto \ab{n}_F$. We will refer to $\fdr{\AA^1}$ as the \textit{$F$-de Rham complex} of the affine line $\AA^1 = \spec R[x]$; we will abusively also refer to the differential $\nabla_F$ as the \textit{$F$-derivative}.
\end{definition}
\begin{example}
\hspace{0in} 
\begin{itemize}
    \item For the additive formal group law, $\ab{n}_F = n$; so the resulting $F$-de Rham complex is simply the usual de Rham complex.
    \item For the multiplicative formal group law, $\ab{n}_F = [n]_q$; so the resulting $F$-de Rham complex is simply the $q$-de Rham complex. Note that the $F$-derivative can be defined directly on polynomials (instead of only on monomials) via $f(x) \mapsto \frac{f(qx) - f(x)}{(q-1) x}$. This can be seen directly from \cref{prop: rational-qiso}: indeed,
    $$\ce_F(\ell_F(t)z) = \frac{(1+t)^z - 1}{t} = \frac{q^z - 1}{q-1},$$
    and the operator $q^{x\partial_x}$ sends $f(x) \mapsto f(qx)$.
\end{itemize}
\end{example}
\begin{remark}\label{rmk: fdr not divided}
One can consider a slight variant of the $F$-de Rham complex, given by the complex
$$C^\bull := \left(R\pw{t}[x] \to R\pw{t}[x] dx\right), \ x^n \mapsto [n]_F x^{n-1} dx.$$
This is the $s$-de Rham complex for the function $s: \Z_{\geq 0} \to R$ sending $n\mapsto [n]_F$. Then, there is a quasi-isomorphism $\fdr{\AA^1} \simeq \eta_t C^\bull$.
\end{remark}

When $\QQ\subseteq R$, there is a general formula for the $F$-derivative.
\begin{prop}\label{prop: rational-qiso}
Suppose that $\QQ\subseteq R$, and let $F$ be a formal group law over $R$. Then there is an equality of $R\pw{t}$-linear operators on $R\pw{t}[x]$:
$$\nabla_F = \frac{1}{tx} \ce_F(\ell_F(t) x\partial_x).$$
In particular, there is a canonical isomorphism $\fdr{\AA^1} \cong \Omega_{\QQ[x]/\QQ} \otimes_\QQ R\pw{t}$.
\end{prop}
\begin{proof}
Let $\nabla_F'$ denote the expression on the right-hand side. By definition of $\nabla_F$, it suffices to check that $\nabla_F'(x^m) = \pdb{m} x^{m-1}$ for every $m\geq 1$. Write $\ce_F(t) = \sum_n a_n t^n$; then
\begin{align*}
    \nabla'_F(x^m) & = \frac{1}{xt} \sum_n a_n \ell_F(t)^n (x\partial_x)^n(x^m) \\
    & = \frac{1}{xt} \sum_n a_n (m\ell_F(t))^n x^m \\
    & = \frac{1}{t} \ce_F(m\ell_F(t)) x^{m-1} = \pdb{m} x^{m-1},
\end{align*}
as desired.
\end{proof}
\begin{remark}\label{rmk: isomorphic fgls have the same dr complex}
Since every formal group law over a $\QQ$-algebra is isomorphic to the additive formal group law, the final statement of \cref{prop: rational-qiso} is a special case of the following more general observation: if $F_1$ and $F_2$ are isomorphic formal group laws, then the associated de Rham complexes are also isomorphic.
\end{remark}
\begin{example}
Using \cref{prop: rational-qiso}, we can make the $F$-derivative explicit for the hyperbolic formal group law. In this case, $\ce_F(t) = \tanh(t)$, so that $\ell_F(t) = \tanh^{-1}(t)$, and
$$\ce_F(\ell_F(t) z) = \tanh(z\tanh^{-1}(t)) = \frac{q^z - (2-q)^z}{q^z+(2-q)^z},$$
where $q-1 = t$.
Since the operator $q^{x\partial_x}$ sends $f(x) \mapsto f(qx)$, the $F$-derivative can be expressed as
$$\nabla_F: f(x) \mapsto \frac{1}{(q-1)x} \frac{f(qx) - f((2-q)x)}{f(qx) + f((2-q)x)}.$$
\end{example}

\begin{example}\label{ex: morava k-theory}
Let $R_0 = \FF_p[v_n]$. Then, there is a unique formal group law (known as the \textit{Honda formal group law}; see \cite{original-honda-fgl}) over $R_0$ which is characterized by the property that its $p$-series is given by $[p]_F(t) = v_n t^{p^n}$. This implies that up to a unit in $R_0\pw{t}$, we have $[m]_F(t) = v_n^{v_p(m)} t^{m^n}$, where $v_p(m)$ denotes the $p$-adic valuation of $m$.
The formal group law over $R_0$ lifts to a formal group law over $R = \Z_p[v_n]$ such that over $R\otimes \QQ \cong \QQ_p[v_n]$, its logarithm is given by
$$\ell_F(x) = \sum_{j\geq 0} v_n^{\frac{p^{jn}-1}{p^n-1}} \frac{x^{p^{jn}}}{p^j}.$$
For example, if $n=1$ and we adjoin a $(p-1)$st root $\beta$ of $v_1$, this is essentially the logarithm of the Artin-Hasse exponential, so that
$$\ell_F(x) = -\frac{1}{\beta} \sum_{p\nmid d} \frac{\mu(d)}{d} \log(1 - (\beta x)^d).$$
One can say something similar for general $n$.
Recall that the polylogarithm is defined by $\plog_s(x) = \sum_{j\geq 1} \frac{x^j}{j^s}$, so that $\plog_1(x) = -\log(1-x)$. If $\beta$ denotes a $(p^n-1)$st root of $v_n$, then $\ell_F(x)$ can be understood as a ``$p^n$-typical'' version of $\frac{1}{\beta} \plog_{1/n}(\beta x)$.

When $n=1$, the resulting $F$-de Rham complex over $\FF_p\pw{t}$ is closely related to the mod $p$ reduction of the $q$-de Rham complex: indeed, observe that the $p$-series of the multiplicative formal group law is congruent to $t^p\pmod{p}$. The resulting lifted formal group law over $\Z_p[v_1]$ is the $p$-typification of the multiplicative formal group law (see \cite[Appendix 2]{green}). For higher $n$, the resulting $F$-de Rham complex behaves qualitatively similar to the case $n=1$. For instance, we have
$$\H^0(\fdr{\AA^1}) \cong \FF_{p^n}[v_n]\pw{t}, \ \H^1(\fdr{\AA^1}) \cong \bigoplus_{j\geq 1} \FF_{p^n}[v_n, t]/(v_n^{v_p(j)} t^{j^n}) \{x^{j-1} dx\}.$$
\end{example}
\begin{example}\label{ex: lubin-tate}
Let $n\geq 1$, let $k$ be an algebraically closed field of characteristic $p>0$, and let $R = W(k)\pw{u_1, \cdots, u_{n-1}}$ denote the Lubin-Tate ring. Let $F$ denote the formal group law associated to the universal deformation of a chosen formal group law of height $n$ over $k$. Then, the resulting $F$-de Rham complex specializes to the $q$-de Rham complex when $n=1$. For general $n$, this $F$-de Rham complex is closely related to deep phenomena in chromatic homotopy theory (see \cref{rmk: homotopy theory}).
\end{example}

\begin{lemma}[$F$-Taylor expansion]\label{lem: F-taylor}
Let $F$ be a formal group law over $R$ such that $[n]_F(t)\in R\pw{t}$ is not a zero-divisor for any $n > 0$. If $f(x)\in (R\otimes \QQ)\pw{t, x-1}$, there is a Taylor expansion
$$f(x) = \sum_{n\geq 0} \nabla_F^n(f(x))|_{x=1} \frac{(x-1)^n_s}{n!_F}.$$
Here, $(x-1)^n_s$ denotes the symbol from \cref{def: (x + y)^n_s} with $y = -1$.
\end{lemma}
\begin{proof}
This is the same argument as in \cite[Proposition 4.4]{anschutz-le-bras}. First, observe that if $g(x) \in (R\otimes \QQ)\pw{t, x-1}$ is a function such that $\nabla_F^n(g(x))|_{x=1} = 0$ for all $n\geq 0$, then $g = 0$. Indeed, since $\nabla_F$ is simply the usual derivative modulo $t$, we see that $g(x)$ is divisible by $t$. Write $g(x) = t g_1(x)$; then, $\nabla_F^n(g_1(x))|_{x=1} = 0$ for all $n\geq 0$, so $t \mid g_1(x)$. Continuing, we see that $g(x)$ is infinitely $t$-divisible, and hence is zero (since $t$ is topologically nilpotent).

We can now apply the above observation to
$$g(x) := f(x) - \sum_{n\geq 0} \nabla_F^n(f(x))|_{x=1} \frac{(x-1)^n_s}{n!_F}.$$
By definition of $(x-1)^n_s$, we know that $\nabla_F(\frac{(x-1)^n_s}{n!_F}) = \frac{(x-1)^{n-1}_s}{(n-1)!_s}$; so $\nabla_F^n(g(x))|_{x=1} = 0$ for all $n\geq 0$, and hence $g = 0$, as desired.
\end{proof}
\begin{corollary}[$F$-logarithm]\label{cor: F-log}
Let $F$ be a formal group law over $R$ such that $[n]_F(t)\in R\pw{t}$ is not a zero-divisor for any $n > 0$. Consider the function $F\log(x) \in (R\otimes \QQ)\pw{t, x-1}$ given by $\frac{t}{\ell_F(t)} \log(x)$. Then, we have:
\begin{enumerate}
    \item $\nabla_F(F\log(x)) = \frac{1}{x}$.
    \item $F\log(xy) = F\log(x) + F\log(y)$.
    \item There is a series expansion
    $$F\log(x) = \sum_{n\geq 1} \frac{\pdb{-n+1}_F \cdots \pdb{-1}_F}{n!_F} (x-1)^n_s.$$
\end{enumerate}
\end{corollary}
\begin{proof}
The first statement follows from \cref{prop: rational-qiso}. Indeed, write $\ce_F(y) = \sum_{n\geq 1} a_n y^n$; the condition that $F(x,y) \equiv x + y\pmod{(x,y)^2}$ forces $a_1 = 1$. Since 
$$(x\partial_x)(F\log(x)) = \frac{t}{\ell_F(t)} (x\partial_x) \log(x) = \frac{t}{\ell_F(t)},$$
we see that
\begin{align*}
    x\nabla_F(F\log(x)) & = \frac{1}{t} \sum_{n\geq 1} a_n \ell_F(t)^n (x\partial_x)^n(F\log(x)) \\
    & = \frac{1}{t} \left(\ell_F(t) \cdot \frac{t}{\ell_F(t)} + \sum_{n\geq 2} a_n \ell_F(t)^n (x\partial_x)^n(F\log(x))\right).
\end{align*}
The second sum vanishes, since $(x\partial_x)^n(F\log(x)) = 0$ for $n\geq 2$. The first term cancels out to give $x\nabla_F(F\log(x)) = 1$, as desired.

The second statement is clear.
For the third statement, observe that
$$\nabla_F^n(F\log(x)) = \nabla_F^{n-1}(1/x) = \pdb{-n+1}_F \cdots \pdb{-1}_F x^{-n}.$$
Evaluating at $x=1$ and using \cref{lem: F-taylor} gives the desired claim.
\end{proof}
\begin{warning}
The $F$-logarithm $F\log(x)$ is \textit{not} the same as the logarithm $\ell_F(x)$ associated to the formal group law. This unfortunate terminology stems from attempting to simultaneously emulate the standard terminology ``$q$-logarithm'' and the ``logarithm of the multiplicative formal group law''.
\end{warning}
\begin{remark}\label{rmk: Flog lives in divided power}
\cref{cor: F-log} implies that $F\log(x)$ is a well-defined class in the ring $R\pw{t}\left[x^{\pm 1}, \frac{(x-1)^n_s}{n!_F}\right]$; this is the ring of functions on an $F$-analogue of the divided power completion of the identity section of $(\GG_m)_{R\pw{t}}$.
\end{remark}
\begin{remark}
The formal series in \cref{cor: F-log}(3) can be written for arbitry GNS $s$; when it exists and converges, its $s$-derivative will formally be $1/x$. However, we have chosen to state \cref{cor: F-log} only in the case of GNS arising via formal group laws, since it is otherwise difficult to get a computational grip on the resulting formal series.
\end{remark}
\begin{example}
When $F$ is the multiplicative formal group law over $\Z$, the function $F\log(x)$ can be identified with the $q$-logarithm
$$\log_q(x) = \sum_{n\geq 1} (-1)^{n+1} q^{-\binom{n}{2}} \frac{(x-1)(x-q) \cdots (x-q^{n-1})}{[n]_q} \in \QQ\pw{q-1, x-1}.$$
Indeed, this follows from \cref{cor: F-log}(3) and the observation that $\pdb{-j}_F = [-j]_q = -q^{-j} [j]_q$.
See \cite[Section 4]{anschutz-le-bras} for more on the $q$-logarithm.
\end{example}

Let us summarize some of the results from the previous section upon specialization to the $F$-de Rham complex:
\begin{theorem}\label{thm: omnibus-fgl}
Let $R$ be a commutative ring, and let $F$ be a formal group law over $R$ such that $[n]_F(t)\in R\pw{t}$ is not a zero-divisor for any $n > 0$. Let $\hat{\GG} = \spf R\pw{t}$ denote the formal group over $R$. Then:
\begin{enumerate}
    \item Let $R\pw{t}\pdb{x}_F$ denote the ring $R\pw{t}[x, \frac{x^n}{[n]_F!}]_{n\geq 0}$. Then the Poincar\'e lemma holds: the cohomology of the complex $\fdr{\AA^1} \otimes_{R\pw{t}[x]} R\pw{t}\pdb{x}_F$ is concentrated in degree zero, where it is isomorphic to $R\pw{t}$.
    \item The Cartier isomorphism holds: there is a canonical isomorphism
    $$\fdr{\AA^{1,(p)}}^i \otimes_{R\pw{t}} R\pw{t}/\pdb{p}_F \cong \H^i(\fdr{\AA^1} \otimes_{R\pw{t}} R\pw{t}/\pdb{p}_F)$$
    sending $(x^p)^n\mapsto x^{np}$ in degree zero and $(x^p)^n d(x^p) \mapsto [x^{np} x^{p-1} dx]$. Note that $\spf R\pw{t}/[p]_F \cong \hat{\GG}[p]$.
    \item The d\'ecalage isomorphism holds: replace $R\pw{t}$ with its $\pdb{p}_F$-adic completion. Let $\varphi: R\pw{t} \to R\pw{t}$ denote the $R$-algebra map sending $t\mapsto [p]_F(t)$, i.e., the map induced on rings by the multiplication-by-$p$ map $\hat{\GG} \to \hat{\GG}$. Then, there is a quasi-isomorphism $\varphi^\ast \fdr{\AA^{1,(p)}} \xar{\sim} \eta_{\pdb{p}_F} \fdr{\AA^1}$.
    \item The hypotheses of \cref{cor: s-divided powers criterion} are satisfied, so that there is a criterion for admitting ``$F$-divided powers''. Namely, replace $R\pw{t}$ by its $(p, \pdb{p}_F)$-adic completion, and suppose that for any $n\geq 0$, $\varphi(n!_F)$ is a non-zero-divisor in $R\pw{t}/\pdb{p}_F$. Let $A$ be a $\pdb{p}_F$-completely flat $R\pw{t}$-algebra equipped with a $R\pw{t}$-linear multiplicative map $\phi: \varphi^\ast A \to A$. If $x\in A$ is an element such that $\varphi(x) = x^p$ and $\pdb{p}_F \mid \varphi(x)$, then $\frac{x^n}{n!_s}\in A$.
\end{enumerate}
\end{theorem}
\begin{proof}
The first part is \cref{prop: s-poincare}. The second part is \cref{prop: s-cartier}, where the hypothesis in the proposition holds because $\pdb{1}_F = 1$. The third (and fourth) part is an application of \cref{prop: decalage}. Note that the first hypothesis holds by construction of $\varphi: R\pw{t} \to R\pw{t}$: indeed, $\varphi$ sends
$$s(n) = \frac{[n]_F(t)}{t} \mapsto \frac{[n]_F([p]_F(t))}{[p]_F(t)} = \frac{[np]_F(t)}{[p]_F(t)} = \frac{s(np)}{s(p)}.$$
The second hypothesis follows from the assumption on $R\pw{t}$.
\end{proof}
\begin{remark}
Using \cref{rmk: isomorphic fgls have the same dr complex}, one can upgrade \cref{thm: omnibus-fgl} to the case when $R$ (rather, $\spec R$) is replaced by the moduli stack of formal groups. However, we will not discuss this further in this article.
\end{remark}
Motivated by \cite[Conjecture 3.1]{scholze-q-def}, we propose the following conjecture. In fact, considerations with ring stacks following \cite{drinfeld-prism} strongly indicate that the conjecture is \textit{false}, but we have stated it nonetheless in the hopes that it might spur investigation into these sort of questions. Arpon Raksit has informed the first author that he is currently working on some variant of this conjecture.
\begin{conjecture}\label{conj: polynomial maps}
Let $\Poly_R$ denote the category of polynomial $R$-algebras and $R$-algebra maps between them. Let $F$ be a formal group law over $R$ such that $[n]_F(t)\in R\pw{t}$ is not a zero-divisor for any $n > 0$. Then, there is a functor $\Gammaf{-}: \Poly_R \to \CAlg(R\pw{t})$ landing in the $\infty$-category of $\Eoo$-$R\pw{t}$-algebras which sends $R[x_1, \cdots, x_n] \mapsto \fdr{\AA^n}$.\footnote{In other words, the assignment $R[x_1, \cdots, x_n] \mapsto \fdr{\AA^n}$ is functorial in $R$-algebra maps of polynomial $R$-algebras.} Furthermore, each part of \cref{thm: omnibus-fgl} admits a generalization to $\Gammaf{-}$.
\end{conjecture}
\begin{remark}
If the preceding conjecture is true, then the construction $\fdr{-}$ can be extended to all (animated) $R$-schemes by left Kan extension:
$$\xymatrix{
\Poly_R^\op \ar[rr]^-{\Gammaf{-}} \ar[d] & & \CAlg(R\pw{t})^\op \\
\Sch_R^\op \ar@{-->}[urr]_-{\Gammaf{-}} & &
}$$
We expect that the resulting functor $X\mapsto \Gammaf{X}$, if it exists, should be rather interesting. When $F$ is the multiplicative formal group law, \cref{conj: polynomial maps} is true, and the resulting assignment $X\mapsto \Gammaf{X}$ is the $q$-de Rham cohomology of \cite{bhatt-scholze}.
\end{remark}

Let us end this section by discussing the motivation behind the construction of the $F$-de Rham complex. We will necessarily be brief, since this is not the main subject of the present article.
\begin{remark}\label{rmk: homotopy theory}
Let $A$ be an even-periodic $\Eoo$-ring equipped with a complex orientation. Then, Quillen defined a canonical formal group law $F(x,y)$ over $R := \pi_0(A)$, whose underlying formal group is $\spf \pi_0(A^{hS^1})$. Let $\tau_{\geq 0} A$ denote the connective cover of $A$, and let $\F^\star_\ev \HP(\tau_{\geq 0} A[x]/\tau_{\geq 0} A)$ denote the even filtration on the periodic cyclic homology of $\tau_{\geq 0} A[x]$ (defined in \cite{even-filtr}). The assignment $A \mapsto (\tau_{\geq 0} A)^{tS^1}$ is the homotopical analogue of the construction of the ``rescaled'' formal group law from \cref{rmk: rescaled-fgl}; in other words, $\spf \pi_0((\tau_{\geq 0} A)^{tS^1})$ is the rescaled analogue of the Quillen formal group $\spf \pi_0(A^{hS^1})$.

Unpublished work of Arpon Raksit shows that the $F$-de Rham complex $\fdr{\AA^1}$ arises as the zeroth associated graded piece $\gr^0_\ev \HP(\tau_{\geq 0} A[x]/\tau_{\geq 0} A)$.\footnote{Using similar methods, one can also show that the variant of the $F$-de Rham complex from \cref{rmk: fdr not divided} arises as the zeroth associated graded piece $\gr^0_\ev \HC^-(A[x]/A)$ in the negative cyclic homology of $A[x]$.} In particular, in this case, the $F$-de Rham complex admits the structure of an $\Eoo$-$R\pw{t}$-algebra. There are homotopical analogues of each part of \cref{thm: omnibus-fgl}: for example, the d\'ecalage isomorphism of \cref{thm: omnibus-fgl}(3) is proved as \cite[Proposition 3.5.3]{thh-xn}.

When $A = \KU$ is periodic complex K-theory (so $\tau_{\geq 0} A = \ku$ is connective complex K-theory), the formal group law over $\pi_0(A)$ is precisely the multiplicative one; so the $q$-de Rham complex arises as $\gr^0_\ev \HP(\ku[x]/\ku)$. As explained in \cite{tp-Z}, the $p$-completion of $\HP(\ku[x]/\ku)$ can be understood via the topological negative cyclic homology of $\Z_p[\zeta_p][x]$; this is a homotopical analogue of the Bhatt-Scholze construction \cite{bhatt-scholze} of $q$-de Rham cohomology via prismatic cohomology.

When $A = E_n$ is the Morava E-theory associated to the Lubin-Tate formal group (see \cref{ex: lubin-tate}) and $\tau_{\geq 0} A = e_n$ is its connective cover, the $F$-de Rham complex of \cref{ex: lubin-tate} arises as $\gr^0_\ev \HP(e_n[x]/e_n)$. The periodic cyclic homology $\HP(e_n[x]/e_n)$ plays an important role in higher chromatic analogues of the work of Bhatt-Morrow-Scholze \cite{bms-ii}, and will be explored in future work.

Although this does not quite fall into the above framework, the first author hopes to show in future work that when $A = L^s(\Z)$ is the symmetric L-theory of the integers (see \cite{l-theory-of-Z}), the $F$-de Rham complex associated to the hyperbolic formal group law is closely related to $\gr^0_\ev \HP(L^s(\Z)[x]/L^s(\Z))$. This is a manifestation of the observation (going back to the Hirzebruch signature theorem) that the logarithm of the formal group law associated to the complex orientation on $L^s(\Z)$ is given by the hyperbolic tangent function $\tanh(x)$.
\end{remark}
\subsection{A variant of the Weyl algebra}
\label{sec: weyl-algebra}

It is well-known that the classical de Rham complex over a base commutative ring $R$ is Koszul dual to the usual Weyl algebra of differential operators on the affine line:
$$\cd_{\AA^1} = R \pdb{x, \partial_x}/([\partial_x,x] = 1).$$
The complex of \cref{rmk: fdr not divided} for the additive formal group is also Koszul dual to a rescaled analogue $R\pw{t} \pdb{x, D}/([D,x] = t)$ of this Weyl algebra; this rescaling amounts to replacing $\partial_x$ by $t\partial_x$.
This fact has an analogue for arbitrary formal group laws, as we now explain. It turns out to be significantly more convenient to study the Weyl algebra of $\GG_m$ instead, so we will restrict to that case. Some of the discussion in this section appears briefly in \cite[Section 3.3]{coh-gr}. As usual, fix a base commutative ring $R$ and a formal group law $F(x,y)$ over $R$.
\begin{definition}\label{def: weyl algebra}
The \textit{$F$-Weyl algebra} of $\GG_m = \spec R[x^{\pm 1}]$ is defined to be the associative $R\pw{t}$-algebra given by
$$\fdiff{\GG_m} := R\pw{t}\pdb{x^{\pm 1},y}^\wedge_y/(yx = xF(y,t)).$$
\end{definition}
\begin{example}
For the additive formal group law, the relation imposed in \cref{def: weyl algebra} is just $yx = x(y+t)$, or equivalently that $[y,x] = tx$. This defines an isomorphism between $\fdiff{\GG_m}$ and the rescaled Weyl algebra for $\GG_m$ by sending $y$ to the rescaled vector field $tx\partial_x$.
\end{example}
\begin{example}
For the multiplicative formal group law, let us write $q = 1+t$. If we define $\tilde{y} = 1 + y$, the relation imposed in \cref{def: weyl algebra} is just 
$$(\tilde{y}-1)x = x(q\tilde{y}-1),$$
or equivalently that $\tilde{y}x = qx\tilde{y}$. Observe that $\tilde{y}$ acts as the operator $q^{x\partial_x}$, so that we obtain an isomorphism between $\fdiff{\GG_m}$ and the completion of $R\pw{q-1}\pdb{x^{\pm 1}, q^{x\partial_x}}/(q^{x\partial_x}x = qxq^{x\partial_x})$ at the ideal $(q^{x\partial_x}-1)$. This algebra is essentially the $q$-Weyl algebra of $\GG_m$, and is sometimes known as the ``quantum torus'', as well as the algebra of $q$-difference operators on the torus.
\end{example}
\begin{lemma}
The algebra $R\pw{t}[x^{\pm 1}]$ is canonically a left $\fdiff{\GG_m}$-module, where the action of $y$ sends $x^n \mapsto [n]_F x^n$.
\end{lemma}
\begin{proof}
Observe that
$$(yx) (x^n) = [n+1]_F x^{n+1} = x F([n]_F, t) x^n = xF(y,t) x^n,$$
so that the action prescribed above satisfies the relation in \cref{def: weyl algebra}.
\end{proof}
The following result describing the complex from \cref{rmk: fdr not divided} as Koszul dual to the $F$-Weyl algebra is sketched in \cite[Proposition 3.3.9]{coh-gr}. We will not need this result below, so we only state it for completeness.
\begin{prop}\label{prop: weyl koszul dual}
The derived tensor product $R\pw{t}[x^{\pm 1}] \otimes_{\fdiff{\GG_m}} R\pw{t}[x^{\pm 1}]$ can be identified with the complex 
$$C^\bull \otimes_{R[x]} R[x^{\pm 1}] \cong \left(R\pw{t}[x^{\pm 1}] \to R\pw{t}[x^{\pm 1}] d\log(x)\right), \ x^n \mapsto [n]_F x^n d\log(x).$$
\end{prop}
\begin{remark}\label{rmk: affine grassmannian}
In \cite[Section 3.3]{coh-gr}, we show that if $R$ is a complex-oriented even-periodic $\Eoo$-ring and $F(x,y)$ is the formal group law over $\pi_0(R)$, then $\fdiff{\GG_m}$ arises as the loop-rotation equivariant homology $\pi_\ast (R[\Omega T]^{h(T \times S^1_\rot)})$ where $T$ is a compact torus of rank $1$ (i.e., a circle, not to be confused with the loop-rotation circle).\footnote{In fact, the first author initially came across the $F$-Weyl algebra in this manner, and this article was originally intended to be about this algebra. However, upon learning of the $F$-de Rham complex from Arpon Raksit, it became clear that the $F$-de Rham complex was a much simpler object to work with; hence the present form of the article.} We also explained that, under the discussion of \cref{rmk: homotopy theory}, the Koszul duality of \cref{prop: weyl koszul dual} is a manifestation of the Koszul duality between $R[\Omega T]^{h(T \times S^1_\rot)}$ and the negative cyclic homology $R[\cL T_+]^{hS^1_\rot} = \HC^-(R[\Omega T]/R)$.
\end{remark}
\begin{remark}[Mellin transform]
Let $\AA^1_{R\pw{t}}$ denote the affine line over $R\pw{t}$ with coordinate $y$, so that $\Z$ acts on $\AA^1_{R\pw{t}}$ via the map $y\mapsto F(y,t)$.
It follows from \cref{def: weyl algebra} and Morita theory that the category of $\fdiff{\GG_m}$-modules is equivalent to the category of quasicoherent sheaves on $\AA^1_{R\pw{t}}/\Z$ which are $y$-complete. In the case of the additive formal group law, this is a $t$-deformation of the Mellin transform, which gives an equivalence between the category of $\cd_{\GG_m}$-modules and quasicoherent sheaves on $\AA^1/\Z$ (where $\Z$ acts by $y\mapsto y+1$).
\end{remark}
Let us now describe some special properties of the center of $\fdiff{\GG_m}$.
\begin{recall}
One property satisfied by the ordinary Weyl algebra in characteristic $p>0$ is that it has a large center: namely, if $R$ is an $\FF_p$-algebra, there is an isomorphism $Z(\cd_{\GG_m}) \cong R[x^{\pm p}, x^p \partial_x^p]$ which identifies $Z(\cd_{\GG_m})$ with the ring of functions on the cotangent bundle of the Frobenius twist $(\GG_m)^{(p)}$. Note that $x^p \partial_x^p \equiv (x\partial_x)^p - x\partial_x \pmod{p}$.
Under the Koszul duality between $\cd_{\GG_m}$ and the de Rham complex, this identification of $Z(\cd_{\GG_m})$ is in fact Koszul dual to the Cartier isomorphism $\H^\ast(\Omega^\bull_{\GG_m/R}) \cong \Omega^\ast_{(\GG_m)^{(p)}/R}$.
\end{recall}
It is therefore natural to ask for a description of the center of $\fdiff{\GG_m}$; this leads to the following result, which is Koszul dual to the Cartier isomorphism of \cref{thm: omnibus-fgl}(b).
\begin{theorem}\label{thm: center Fdiff}
The center of $\fdiff{\GG_m}\otimes_{R\pw{t}} R\pw{t}/\pdb{p}_F$ can be identified as follows:
$$Z(\fdiff{\GG_m}\otimes_{R\pw{t}} R\pw{t}/\pdb{p}_F) \cong R\pw{t}\left[x^{\pm p}, \prod_{j=0}^{p-1} (y +_F [j]_F) \right]/\pdb{p}_F.$$
\end{theorem}
\begin{proof}
Observe that
$$yx^p = x^p F(y,[p]_F) \equiv x^p y\pmod{\pdb{p}_F},$$
so that $x^p$ is in the center of $\fdiff{\GG_m}\otimes_{R\pw{t}} R\pw{t}/\pdb{p}_F$.
Similarly, since $\prod_{j=0}^{p-1} (y +_F [j]_F) \equiv \prod_{j=0}^{p-1} (y +_F [j]_F) \pmod{\pdb{p}_F}$, we have
$$\prod_{j=0}^{p-1} (y +_F [j]_F) x = x \prod_{j=1}^p (y +_F [j]_F) \equiv x \prod_{j=0}^{p-1} (y +_F [j]_F) \pmod{\pdb{p}_F},$$
and hence $\prod_{j=0}^{p-1} (y +_F [j]_F)$ is in the center of $\fdiff{\GG_m}\otimes_{R\pw{t}} R\pw{t}/\pdb{p}_F$. This defines an inclusion 
$$R\pw{t}\left[x^p, \prod_{j=0}^{p-1} (y +_F [j]_F) \right]/\pdb{p}_F \subseteq Z(\fdiff{\GG_m}\otimes_{R\pw{t}} R\pw{t}/\pdb{p}_F),$$
which can be checked to be an isomorphism.
\end{proof}
\begin{example}
For the additive formal group law over $\Z$ (say), we have
$$\prod_{j=0}^{p-1} (y + jt) \equiv y^p - t^{p-1} y \pmod{p},$$
so that upon identifying $y = x\partial_x$, \cref{thm: center Fdiff} implies the following isomorphism for the rescaled Weyl algebra:
$$Z(\fdiff{\GG_m}/p) \cong \FF_p\pw{t}[x^{\pm p}, (x\partial_x)^p - t^{p-1} x\partial_x].$$
Regarding the above algebra as a graded ring with both $t$ and $x\partial_x$ in weight $1$ allows us to replace $t$ by a polynomial generator (instead of a power series variable). Inverting $t$ and taking $\GG_m$-invariants (which is to be thought of as setting $t=1$) then produces the ring $\FF_p[x^{\pm p}, t^{-p} x^p \partial_x^p] = \co_{T^\ast \GG_m^{(1)}}$.
\end{example}
\begin{example}
For the multiplicative formal group law over $\Z$ (say), if we define $\tilde{y} = y+1$ as above (so that $\tilde{y} = q^{x\partial_x}$), we have
$$\prod_{j=0}^{p-1} (y +_F (q^j-1)) = \prod_{j=0}^{p-1} (q^j \tilde{y}-1) \equiv q^{p(p-1)/2} (\tilde{y}^p - 1) \pmod{[p]_q}.$$
Note that $q^{p(p-1)/2}$ is $1\pmod{[p]_q}$ for $p>2$, but is $-1\pmod{[2]_q}$. In either case, $q^{p(p-1)/2}$ is a unit, so \cref{thm: center Fdiff} implies the following isomorphism:
$$Z(\fdiff{\GG_m}/[p]_q) \cong \Z[\zeta_p][x^{\pm p}, q^{p x \partial_x}].$$
\end{example}
\begin{remark}
Since $\fdiff{\GG_m}$ can be recovered from the equivariant homology $\pi_\ast (R[\Omega T]^{h(T \times S^1_\rot)})$ for a compact torus $T$ of rank $1$ (see \cref{rmk: affine grassmannian}), it is natural to wonder whether there is an explanation of \cref{thm: center Fdiff} from the perspective of the affine Grassmannian. In the case when $R$ is ordinary (integral) homology or K-theory, this has been answered in \cite{lonergan-steenrod} --- in fact, the methods there are sufficiently geometric that they work even for more general $R$ (and for $T$ replaced by a more general connected compact Lie group!), so we refer the reader to \textit{loc. cit.} for further discussion of this question.
\end{remark}
\subsection{An analogue of the Bhatt-Lurie Cartesian square}
\label{sec: bhatt-lurie-analogue}

Throughout this section, we will fix a $p$-completely flat $\Z_p$-algebra $R$ and a formal group law $F(x,y)$ over $R$ (so $R$ is torsionfree). The symbol $R\pw{t}$ will always denote the \textit{$p$-adic completion} of the formal power series ring, and all constructions will be done internal to the category of $(p,t)$-adically $R\pw{t}$-schemes. (We have omitted the completion from the notation for readability.) Let $\hat{\GG}$ denote the associated formal group, so that its underlying formal scheme is $\spf R\pw{t}$.

In \cite[Lemma 3.5.18]{apc}, Bhatt-Lurie showed that there is a Cartesian square of group schemes over $R$:
\begin{equation}\label{eq: gm-sharp-log-square}
    \xymatrix{
    \GG_m^\sharp \ar[r]^-{\log} \ar[d] & \GG_a^\sharp \ar[d]^-{x\mapsto \exp(px)} \\
    \GG_m \ar[r]_-{x\mapsto x^p} & \GG_m^{(1)}.
    }
\end{equation}
Taking vertical quotients, we obtain an isomorphism between $\GG_m/\GG_m^\sharp$ and $\GG_m^{(1)}/\GG_a^\sharp$; one can identify $\GG_m/\GG_m^\sharp$ with the ``de Rham stack'' $\GG_m^\dR$ of $\GG_m$ (see, e.g., \cite{bhatt-f-gauge-lectures}), so that this isomorphism describes $\GG_m^\dR$ in terms of the group scheme $\GG_a^\sharp$.

The proof of \cref{eq: gm-sharp-log-square} in \textit{loc. cit.} used the relationship between the group schemes appearing in the square and the ring scheme of Witt vectors.
In this section, we prove an $F$-analogue of this result (see \cref{thm: BL-analogue}); in the case of the additive formal group law, this reproves \cref{eq: gm-sharp-log-square}. Let us state at the outset that in the case when $F$ is the multiplicative formal group law, this result was obtained in a discussion between the first author and Michael Kural. Moreover, the argument in this section rests crucially on \cref{eq: basis expression for betan-Flogy}, the $q$-analogue of which (\cref{ex: michael's identity}) was proved by Michael Kural. Any errors below are solely the fault of the first author!
\begin{definition}
\cref{rmk: rescaled-fgl} gives a formal group $\hat{\GG}_t$ over $\spf R\pw{t}$ whose logarithm is $\tilde{\ell}_F(x) = \frac{1}{t} \ell_F(tx)$. Let $x$ denote the coordinate on $\hat{\GG}_t$, so that its underlying formal scheme is $\spf R\pw{t,x}$. Let $\hat{\GG}_t^\vee$ denote the Cartier dual $\Hom(\hat{\GG}_t, (\GG_m)_{R\pw{t}})$ of $\hat{\GG}_t$; see \cite[Section 3]{drinfeld-formal-group} for some generalities on Cartier duals of formal groups.
The element $x\in \co_{\hat{\GG}_t}$ defines a homomorphism $\tau: \hat{\GG}_t^\vee \to (\GG_a)_{R\pw{t}}$.
\end{definition}
\begin{observe}\label{obs: beta_n notation}
Over $(R\otimes \QQ)\pw{t}$, the rescaled logarithm $\tilde{\ell}_F$ of \cref{rmk: rescaled-fgl} defines an isomorphism $\tilde{\ell}_F: \hat{\GG}_t \xar{\sim} (\hat{\GG}_a)_{(R\otimes \QQ)\pw{t}}$ of formal groups. Therefore, the canonical pairing $\hat{\GG}_t \times_{R\pw{t}} \hat{\GG}_t^\vee \to (\GG_m)_{R\pw{t}}$ fits into a diagram
$$\xymatrix{
\hat{\GG}_t \times_{R\pw{t}} \hat{\GG}_t^\vee \ar[dr]^-\mu \ar[d]^-\sim_-{\tilde{\ell}_F\times \id} & \\
(\hat{\GG}_a)_{(R\otimes \QQ)\pw{t}} \times_{R\pw{t}} \hat{\GG}_t^\vee \ar[r]_-{\nu} &
(\GG_m)_{(R\otimes \QQ)\pw{t}}.}$$
Since $R\pw{t}$ is $(p,t)$-adically complete, the Cartier dual of $(\hat{\GG}_a)_{R\pw{t}}$ can be identified with the divided power completion $(\GG_a^\sharp)_{R\pw{t}}$. The pairing $\nu$ is base-changed from $R\pw{t}$ itself, where it is given by the formula
$$\nu: (x,y) \mapsto \exp(xy).$$
It follows that the pairing $\mu$ is given by
$$\mu(x,y) = \exp(\tilde{\ell}_F(x) y).$$
This can be expanded as a power series in $x$:
$$\mu(x,y) = \sum_{n\geq 0} \beta_n(y) x^n.$$
Unwinding the definition of the Cartier dual, and using that $R\pw{t}$ is $p$-torsionfree (using our assumption that $R$ is a $p$-completely flat $\Z_p$-algebra), we see that the ring of functions on $\hat{\GG}_t^\vee$ has underlying $R\pw{t}$-module given by (the $(p,t)$-adic completion of)
$$\co_{\hat{\GG}_t^\vee} = R\pw{t}\{\beta_n(y)\}_{n\geq 0}.$$
\end{observe}
\begin{example}\label{ex: gm-cartier-dual}
When $F$ is the multiplicative formal group law, the function $\mu$ is simply
$$\mu(x,y) = \exp\left(\frac{y}{q-1} \log(1+(q-1)x)\right) = (1+(q-1)x)^{y/(q-1)};$$
its power series expansion is given by
$$\mu(x,y) = \sum_{n\geq 0} \frac{\prod_{j=0}^{n-1} (y - j(q-1))}{n!} x^n.$$
This expression plays an important role in \cite{drinfeld-formal-group}.
\end{example}
\begin{example}
When $F$ is the hyperbolic formal group law (so that $\ell_F(x) = \tanh^{-1}(x) = \frac{1}{2} \log\left(\frac{1+x}{1-x}\right)$), the function $\mu$ is
$$\mu(x,y) = \exp\left(\frac{y}{2(q-1)} \log\left(\frac{1+(q-1)x}{1-(q-1)x}\right)\right) = \left(\frac{1+(q-1)x}{1-(q-1)x}\right)^{y/2(q-1)}.$$
The power series expansion of this function is somewhat complicated: one can show that upon writing $\mu(x,y) = \sum_{n\geq 0} \beta_n(y) x^n$, we have $\beta_0 = 1$, $\beta_1 = y$, and there is a recurrence
$$\beta_{n+2}(y) = \frac{y\beta_{n+1}(y) + n(q-1)^2 \beta_n(y)}{n+2}.$$
Using a computer, one can compute that the first few terms of this expansion are
\begin{align*}
    \mu(x,y) & = 1 + y x + \frac{y^2}{2} x^2 + \frac{2 (q-1)^2 y + y^3}{3!} x^3 + \frac{8 (q-1)^2 y^2 + y^4}{4!} x^4 \\
    & + \frac{24 (q-1)^4 y + 20 (q-1)^2 y^3 + y^5}{5!} x^5 + \cdots.
\end{align*}
Observe that $\beta_n(y) \equiv \frac{y^n}{n!} \pmod{2}$, reflecting the fact that the base-change of the hyperbolic formal group to $\FF_2$ is isomorphic to the additive formal group.
\end{example}
\begin{definition}\label{def: Gm-sharp F}
Let $\GG_m^{\sharp, F}$ denote the formal scheme over $\spf R\pw{t}$ given by (the $(p,t)$-adic completion of)
$$\GG_m^{\sharp, F} = \spf R\pw{t}\left[y^{\pm 1}, \frac{(y-1)^n_s}{n!_F}\right]_{n\geq 0}.$$
This can be viewed as the ``$F$-divided power hull'' of the identity section of $(\GG_m)_{R\pw{t}}$.
Equip $\GG_m^{\sharp, F}$ with the structure of a group scheme where the coproduct sends $y\mapsto y\otimes y$. It is not immediate that this is well-defined, but we will prove this below in \cref{cor: Gm-sharp is well-defined}. There is a canonical homomorphism $\can: \GG_m^{\sharp, F} \to (\GG_m)_{R\pw{t}}$.

Note that \cref{rmk: Flog lives in divided power} implies that $F\log(y)$ defines an element of the coordinate ring of $\GG_m^{\sharp, F}$, i.e., it defines a map $F\log: \GG_m^{\sharp, F} \to (\GG_a)_{R\pw{t}}$. This is in fact a homomorphism, since $F\log(y_1 y_2) = F\log(y_1) + F\log(y_2)$.
\end{definition}
\begin{prop}\label{prop: michael's result}
Work over the base $(R\otimes \QQ)\pw{t}$. Then, the iterated $F$-derivative of $\mu(x,F\log(y))$ with respect to the variable $y$ is given by
\begin{equation}\label{eq: F-derivative iterated}
    \nabla_{F,y}^n \mu(x,F\log(y)) = \frac{x(x+_{\tilde{F}}\pdb{-1}_F(t)) \cdots (x+_{\tilde{F}}\pdb{-n+1}_F(t))}{y^n} \mu(x,F\log(y)).
\end{equation}
\end{prop}
\begin{proof}
Observe that:
\begin{align*}
    \mu(x, F\log(y)) & = \sum_{n\geq 0} \beta_n(F\log(y)) x^n = \exp(F\log(y) \tilde{\ell}_F(x)) \\
    & = \exp\left(\frac{t}{\ell_F(t)} \log(y) \cdot \frac{\ell_F(tx)}{t}\right) = \exp\left(\log(y) \frac{\ell_F(tx)}{\ell_F(t)}\right) = y^{\frac{\ell_F(tx)}{\ell_F(t)}};
\end{align*}
the third equality used the definition of $F\log(y)$ via \cref{cor: F-log} and the definition of $\tilde{\ell}_F(x)$ via \cref{rmk: rescaled-fgl}. One can deduce \cref{eq: F-derivative iterated} from this; let us illustrate this rather inefficiently. Let us write $a = \frac{\ell_F(tx)}{\ell_F(t)}$ for notational simplicity, so that
$$y^a = \sum_{m\geq 0} \frac{a(a-1) \cdots (a-(m-1))}{m!} (y-1)^m,$$
and $\partial_y y^a = a y^{a-1}$.

We can now inductively compute the iterated $F$-derivative using \cref{prop: rational-qiso}. We begin with the base case $n=1$. 
Note that $(y\partial_y) y^{a} = a y^{a}$, so that
\begin{equation}\label{eq: iterated y-partial_y}
    (y\partial_y)^m y^{a} = a^m y^{a}
\end{equation}
by an easy induction on $m$.
Write $\ce_F(z) = \sum_{m\geq 0} b_m z^m$; then using \cref{eq: iterated y-partial_y}, we have:
\begin{align*}
    \nabla_{F,y} \mu(x, F\log(y)) & = \frac{1}{yt} \sum_{m\geq 0} b_m \ell_F(t)^m (y\partial_y)^m y^{a} = \frac{1}{yt} \sum_{m\geq 0} b_m \ell_F(t)^m a^m y^{a} \\
    & = \frac{y^{a}}{yt} \sum_{m\geq 0} b_m \ell_F(tx)^m = \frac{y^{a}}{yt} \ce_F(\ell_F(tx)) = \frac{tx}{yt} \cdot y^{a} \\
    & = \frac{x}{y} y^{a} = \frac{x}{y} \mu(x, F\log(y)),
\end{align*}
as desired.

The proof of the iterated $F$-derivative is similar. Indeed, note that \cref{eq: iterated y-partial_y} implies that for any $j\geq 0$, we have:
\begin{equation}\label{eq: redux iterated y-partial_y}
    (y\partial_y)^m \left(\frac{\mu(x,F\log(y))}{y^j}\right) = (y\partial_y)^m y^{a-j} = (a-j)^m y^{a-j} = (a-j)^m \frac{\mu(x,F\log(y))}{y^j}.
\end{equation}
Assume that \cref{eq: F-derivative iterated} holds for $n$; then:
\begin{align}
    \nabla_{F,y}^{n+1} \mu(x,F\log(y)) & = \nabla_{F,y} \nabla_{F,y}^n \mu(x,F\log(y)) \nonumber \\
    & = x(x+_{\tilde{F}}\pdb{-1}_F(t)) \cdots (x+_{\tilde{F}}\pdb{-n+1}_F(t)) \nabla_{F,y} \left(\frac{\mu(x,F\log(y))}{y^n}\right). \label{eq: nablaF-n+1}
\end{align}
The derivative on the right-hand side can be calculated as follows:
\begin{align*}
    \nabla_{F,y} \left(\frac{\mu(x,F\log(y))}{y^j}\right) & = \frac{1}{yt} \sum_{m\geq 0} a_m \ell_F(t)^m (y\partial_y)^m \left(\frac{\mu(x,F\log(y))}{y^j}\right) \\
    & = \frac{1}{yt} \sum_{m\geq 0} a_m \ell_F(t)^m \left(\frac{\ell_F(tx)}{\ell_F(t)} - j\right)^m \frac{\mu(x,F\log(y))}{y^j} \\
    & = \frac{\mu(x,F\log(y))}{y^j} \frac{1}{yt} \sum_{m\geq 0} a_m (\ell_F(tx) - j\ell_F(t))^m \\
    & = \frac{\mu(x,F\log(y))}{y^j} \frac{1}{yt} \ce_F(\ell_F(tx) - j\ell_F(t)).
\end{align*}
Since 
$$\ell_F(tx) - j\ell_F(t) = \ell_F(tx) + \ell_F([-j]_F(t)) = \ell_F(tx +_F [-j]_F(t))),$$
this becomes
\begin{align*}
    \nabla_{F,y} \left(\frac{\mu(x,F\log(y))}{y^j}\right) & = \frac{\mu(x,F\log(y))}{y^j} \frac{1}{yt} \ce_F(\ell_F(tx +_F [-j]_F(t))) \\
    & = \frac{\mu(x,F\log(y))}{y^j} \frac{x +_{\tilde{F}} \pdb{-j}_F(t)}{y}.
\end{align*}
Plugging this into \cref{eq: nablaF-n+1}, we get that
$$\nabla_{F,y}^{n+1} \mu(x,F\log(y)) = x(x+_{\tilde{F}}\pdb{-1}_F(t)) \cdots (x+_{\tilde{F}}\pdb{-n+1}_F(t))(x+_{\tilde{F}}\pdb{-n}_F(t)) \frac{\mu(x, F\log(y))}{y^{n+1}},$$
as desired.
\end{proof}
\begin{corollary}\label{cor: Flog has pd}
There is a dotted map (which is a homomorphism over $R\pw{t}$) filling in the following diagram:
$$\xymatrix{ 
& & \hat{\GG}_t^\vee \ar[d]^-\tau \\
\GG_m^{\sharp, F} \ar[rr]_{y\mapsto F\log(y)} \ar@{-->}[urr] & & (\GG_a)_{R\pw{t}}.
}$$
\end{corollary}
\begin{proof}
In the notation of \cref{obs: beta_n notation}, we need to show that $\beta_n(F\log(y))\in \co_{\GG_m^{\sharp, F}}$ for every $n\geq 0$. To prove this, let us work over $(R\otimes \QQ)\pw{t}$, and expand $\mu(x, F\log(y))$ as a power series in $\frac{(y-1)^n_s}{n!_F}$ using \cref{lem: F-taylor}. Evaluating \cref{eq: F-derivative iterated} in \cref{prop: michael's result} at $y=1$, we obtain
$$\nabla_{F,y}^n \mu(x,F\log(y)) |_{y=1} = x(x+_{\tilde{F}}\pdb{-1}_F(t)) \cdots (x+_{\tilde{F}}\pdb{-n+1}_F(t)).$$
It follows from \cref{lem: F-taylor} that
\begin{align}
    \sum_{n\geq 0} \beta_n(F\log(y)) x^n & = \mu(x, F\log(y)) \nonumber\\
    & = \sum_{n\geq 0} x(x+_{\tilde{F}}\pdb{-1}_F(t)) \cdots (x+_{\tilde{F}}\pdb{-n+1}_F(t)) \frac{(y-1)^n_s}{n!_F}. \label{eq: basis expression for betan-Flogy}
\end{align}
Taking the coefficient of $x^n$ on the right-hand side expresses $\beta_n(F\log(y))$ as an $(R\otimes \QQ)\pw{t}$-linear combination of the divided powers $\frac{(y-1)^n_s}{n!_F}$; but since no rational denominators appear, this in fact expresses $\beta_n(F\log(y))$ as an $R\pw{t}$-linear combination of the divided powers $\frac{(y-1)^n_s}{n!_F}$, as desired.
\end{proof}
\begin{example}\label{ex: michael's identity}
When $F$ is the multiplicative formal group law, \cref{eq: basis expression for betan-Flogy} reduces to the following identity:
\begin{align}
    \sum_{n\geq 0} \frac{\log_q(y)(\log_q(y)-(q-1)) \cdots (\log_q(y)-(n-1)(q-1))}{n!} x^n \nonumber\\
    = \sum_{n\geq 0} q^{-\binom{j}{2}} x(x-[1]_q) \cdots (x-[n-1]_q) \frac{(y-1)(y-q)\cdots(y-q^{n-1})}{[n]_q!}. \label{eq: michael}
\end{align}
Indeed, we have
$$x+_{\tilde{F}}[-n]_q = x + \frac{q^{-n}-1}{q-1} + (q^{-n}-1)x
= q^{-n} x + [-n]_q = q^{-n}(x - [n]_q).$$
The identity \cref{eq: michael} was communicated to the first author by Michael Kural, and was motivation for the more general \cref{eq: basis expression for betan-Flogy}.
\end{example}
\begin{corollary}\label{cor: Gm-sharp is well-defined}
The group structure on $\GG_m^{\sharp, F}$ is well defined.
\end{corollary}
\begin{proof}
Suppose that $y_1$ and $y_2$ both admit $F$-divided powers $\frac{(y-1)^n_s}{n!_F}$; we need to show that the same is true of the product $y_1 y_2$. Since $F\log(y_1 y_2) = F\log(y_1) + F\log(y_2)$, one can express $\beta_n(F\log(y_1 y_2))$ in terms of $\beta_n(F\log(y_1))$ and $\beta_n(F\log(y_2))$. Using the identity \cref{eq: basis expression for betan-Flogy} in the proof of \cref{cor: Flog has pd} shows that $y_1 y_2$ must also admit $F$-divided powers, as desired.
\end{proof}
The main result of this section is the following, which recovers \cref{eq: gm-sharp-log-square} when $F$ is the additive formal group law. Using \cref{rmk: isomorphic fgls have the same dr complex}, one can in fact refine \cref{thm: BL-analogue} to remain true when the base $R\pw{t}$ (or rather $\hat{\GG} = \spf R\pw{t}$) is replaced by the universal formal group $\hat{\GG}^\univ$ over the moduli stack of formal groups over $p$-nilpotent rings.
\begin{theorem}\label{thm: BL-analogue}
There is a Cartesian square of group schemes over $R\pw{t}$:
\begin{equation}\label{eq: F-analogue BL}
    \xymatrix{
    \GG_m^{\sharp,F} \ar[r]^-{y \mapsto F\log(y)} \ar[d]_-\can & \hat{\GG}_t^\vee \ar[d]^-{\pdb{p}^\ast} \\
    (\GG_m)_{R\pw{t}} \ar[r]_-{y\mapsto y^p} & (\GG_m^{(1)})_{R\pw{t}}.
    }
\end{equation}
The right-vertical map is Cartier dual to the homomorphism $p\ul{\Z} \to \hat{\GG}_t$ sending $p\mapsto \pdb{p}_F(t)$. In particular, there is an extension
$$0 \to (\mu_p)_{R\pw{t}} \to \GG_m^{\sharp,F} \xar{F\log} \hat{\GG}_t^\vee \to 0.$$
\end{theorem}
\begin{proof}
To check that the diagram commutes, we need to check that there is an equality of elements of $\co_{\GG_m^{\sharp,F}}$:
$$y^p = \pdb{p}^\ast(F\log(y)).$$
Since $R$ is $p$-completely flat over $\Z_p$, there is an injection $R\pw{t} \subseteq (R\otimes \QQ)\pw{t}$; so it suffices to check the desired identity in $(R\otimes \QQ)\pw{t}\{\beta_n(y)\}_{n\geq 0}$.
By the discussion in \cref{obs: beta_n notation}, $\pdb{p}^\ast$
can be expressed as
\begin{align}
    \pdb{p}^\ast(z) & = \exp(z\tilde{\ell}_F(\pdb{p}_F(t))) = \exp\left(z\frac{\ell_F(t \pdb{p}_F(t))}{t}\right) \nonumber \\
    & = \exp\left(z\frac{\ell_F([p]_F(t))}{t}\right) = \exp\left(p\frac{\ell_F(t)}{t} z\right). \label{eq: p-ast equation}
\end{align}
Note that this is also $\exp\left(\frac{\ell_F([p]_F(t))}{t} z\right)$.
It follows that
\begin{align*}
    \pdb{p}^\ast(F\log(y)) & = \exp\left(p\frac{\ell_F(t)}{t} F\log(y)\right) = \exp\left(p\frac{\ell_F(t)}{t} \frac{t}{\ell_F(t)}\log(y)\right) = \exp(p\log(y)) = y^p,
\end{align*}
as desired.

To check that the square is Cartesian, first note that the horizontal maps are surjective. This is clear for the Frobenius on $(\GG_m)_{R\pw{t}}$. For the map $F\log$, define $F\exp(z) := \exp(\frac{\ell_F(t)}{t} z)$, so that $F\exp(z) = \sum_{n\geq 0} \beta_n(z)$. There is a homomorphism $\hat{\GG}_t^\vee \to (\GG_m)_{R\pw{t}}$ sending $z\mapsto F\exp(z) := \exp(\frac{\ell_F(t)}{t} z)$, and $z = F\log(F\exp(z))$. Using \cref{eq: basis expression for betan-Flogy} with $y = F\exp(z)$, one sees that $F\exp$ lands in $\GG_m^{\sharp,F}$, i.e., that $\frac{(F\exp(z)-1)^n_s}{n!_F}$ is well-defined in $\co_{\hat{\GG}_t^\vee}$. This implies that $F\log$ is surjective. 

It remains to show that the kernel of $F\log: \GG_m^{\sharp,F} \to \hat{\GG}_t^\vee$ is isomorphic to $(\mu_p)_{R\pw{t}}$. Observe that $F\log(y) = 0$ implies that $\log(y) = 0$, which happens (by the Cartesian square \cref{eq: gm-sharp-log-square}) if and only if $y^p = 1$. Conversely, if $y^p = 1$, then
$$p\cdot F\log(y) = F\log(y^p) = 0,$$
which implies that $F\log(y) = 0$.
\end{proof}
\begin{example}\label{ex: drinfeld-square}
It follows from \cref{thm: BL-analogue} that $\GG_m^{\sharp,F}$ is an extension of $\hat{\GG}_t^\vee$ by $(\mu_p)_{R\pw{t}}$. In the case of the multiplicative formal group law over $R = \Z_p$, this was studied in \cite{drinfeld-formal-group}. Namely, in \cite[Section 5.3.1]{drinfeld-formal-group}, it is shown that there is an extension $\tilde{G}_Q$ of $(\hat{\GG}_{m,q-1})^\vee$ by $(\mu_p)_{\Z_p\pw{q-1}}$, given by the functor
$$\tilde{G}_Q: R \mapsto \{(q,x,u)\in R^\times \times W(R) \times R^\times \mid q-1\text{ is nilpotent}, 1+\Phi_p([q])x = [u^p]\}.$$
Here, $W(R)$ denotes the ring of $p$-typical Witt vectors of $R$.
Drinfeld shows that the group scheme $\tilde{G}_Q$ is isomorphic over $\Z_p\pw{q-1}$ to $\GG_m^{\sharp,F}$ (as extensions of $(\hat{\GG}_{m,q-1})^\vee$ by $(\mu_p)_{\Z_p\pw{q-1}}$). 

As shown in \cite[Appendix D]{drinfeld-formal-group} (see also \cite[Remark C.3]{thh-xn} and \cref{ex: gm-cartier-dual}), the Cartier dual $(\hat{\GG}_{m,q-1})^\vee$ can be identified with
$$(\hat{\GG}_{m,q-1})^\vee = \spf \Z_p\pw{q-1}\left[y, \frac{\prod_{j=0}^{n-1} (y - j(q-1))}{n!}\right]_{n\geq 0}.$$
By \cref{eq: p-ast equation}, the homomorphism $\pdb{p}^\ast$ corresponds to the invertible element
$$\pdb{p}^\ast(z) = \exp\left(p\frac{\log(q)}{q-1} z\right) = q^{pz/(q-1)};$$
this element plays an important role in \cite{drinfeld-formal-group}. Note that this can alternatively be written as
$$\pdb{p}^\ast(z) = \sum_{n\geq 0} \frac{\prod_{j=0}^{n-1} (z - j(q-1))}{n!} [p]_q^n = \sum_{n\geq 0} \frac{\prod_{j=0}^{n-1} (pz - j(q-1))}{n!}.$$
In this case, \cref{thm: BL-analogue} therefore specializes to give a Cartesian square over $\Z_p\pw{q-1}$:
$$\xymatrix{
\tilde{G}_Q \ar[r]^\sim & \GG_m^{\sharp,F} \ar[r]^-{y \mapsto \log_q(y)} \ar[d]_-\can & \hat{\GG}_{m,q-1}^\vee \ar[d]^-{z\mapsto q^{pz/(q-1)}} \\
& (\GG_m)_{\Z_p\pw{q-1}} \ar[r]_-{y\mapsto y^p} & (\GG_m^{(1)})_{\Z_p\pw{q-1}}.
}$$
\end{example}
\begin{remark}
In \cite[Section 5]{drinfeld-formal-group} (in particular, \cite[Remark 5.7.5]{drinfeld-formal-group}), it is shown that $(\hat{\GG}_{m,q-1})^\vee$ is isomorphic to the group scheme $W_{\Z_p\pw{q-1}}[F - (1 + [q] + \cdots + [q]^{p-1})]$, where $W_{\Z_p\pw{q-1}}$ is the Witt scheme over $\Z_p\pw{q-1}$ and $[q] \in W(\Z_p\pw{q-1})$ is the Teichm\"uller lift of $q$. For a more general formal group $\hat{\GG}_t$, the methods of this section might give a Witt vector model for the Cartier dual $\hat{\GG}_t^\vee$, but we have not explored this direction.
\end{remark}
\begin{example}
If $F$ is the hyperbolic formal group law, then $\ell_F(t) = \frac{1}{2} \log\left(\frac{t+1}{1-t}\right)$. Setting $t = q-1$, \cref{eq: p-ast equation} says that
$$\pdb{p}^\ast(z) = \exp\left(p\frac{\log(\frac{q}{2-q})}{2(q-1)} z\right) = \left(\frac{q}{2-q}\right)^{pz/2(q-1)}.$$
The denominator of $2$ appearing in the exponent implies that the cases $p>2$ and $p=2$ behave markedly differently.
Since $\frac{q}{2-q} = 1+(q-1) \cdot \frac{2}{2-q}$, one can alternatively write
$$\pdb{p}^\ast(z) = \sum_{n\geq 0} \frac{\prod_{j=0}^{n-1} (pz - 2j(q-1))}{n!} \left(\frac{1}{2-q}\right)^n = \sum_{n\geq 0} \frac{\prod_{j=0}^{n-1} (z - 2j(q-1))}{n! 2^n} \frac{[p]_q^n}{(2-q)^{pn}}.$$
Write $\hat{\GG}_{h, q-1}$ to denote the formal group over $\Z_p\pw{q-1}$ corresponding to the rescaled hyperbolic formal group law. 
In this case, \cref{thm: BL-analogue} specializes to give a Cartesian square over $\Z_p\pw{q-1}$:
$$\xymatrix{
\GG_m^{\sharp,F} \ar[r]^-{y \mapsto F\log(y)} \ar[d]_-\can & \hat{\GG}_{h, q-1}^\vee \ar[d]^-{z\mapsto \left(\frac{q}{2-q}\right)^{pz/2(q-1)}} \\
(\GG_m)_{\Z_p\pw{q-1}} \ar[r]_-{y\mapsto y^p} & (\GG_m^{(1)})_{\Z_p\pw{q-1}}.
}$$
\end{example}
\begin{example}
Let $F$ be the formal group law over $\Z_p$ given by specializing the formal group law of \cref{ex: morava k-theory} to $v_n = 1$. Let $\hat{\GG} = \spf \Z_p\pw{t}$ denote the associated formal group, and let $\hat{\GG}_t$ denote the rescaled formal group over $\Z_p\pw{t}$. Then, \cref{eq: p-ast equation} says that $\pdb{p}^\ast(z)$ is a ``$p^n$-typical'' version of $\exp\left(p\frac{\plog_{1/n}(t)}{t} z\right)$. If we write $E_p(t)$ to denote the Artin-Hasse exponential, then $\pdb{p}^\ast(z)$ can be made explicit when $n=1$:
$$\pdb{p}^\ast(z) = E_p(t)^{pz/t} = \prod_{p \nmid m} (1-t^m)^{-\frac{p\mu(m)z}{tm}}.$$
If we write $t = \ptl$ (to keep with the notation of \cite{apc}), then \cref{thm: BL-analogue} in this case specializes to give a Cartesian square over $\Z_p\pw{\ptl}$:
$$\xymatrix{
\GG_m^{\sharp,F} \ar[r]^-{y \mapsto F\log(y)} \ar[d]_-\can & \hat{\GG}_\ptl^\vee \ar[d]^-{z\mapsto E_p(\ptl)^{pz/\ptl}} \\
(\GG_m)_{\Z_p\pw{\ptl}} \ar[r]_-{y\mapsto y^p} & (\GG_m^{(1)})_{\Z_p\pw{\ptl}}.
}$$
This can be viewed as a $p$-typical version of \cref{ex: drinfeld-square}.
\end{example}

\begin{remark}
The canonical homomorphism $\can: \GG_m^{\sharp,F} \to (\GG_m)_{R\pw{t}}$ of group schemes defines a quotient group stack $\GG_m^{F\dR} := (\GG_m)_{R\pw{t}}/\GG_m^{\sharp,F}$ over $R\pw{t}$. One can prove by direct calculation that the derived global sections of the structure sheaf of the stack $\GG_m^{F\dR}$ calculates the $F$-de Rham complex $\fdr{\GG_m} := \fdr{\AA^1} \otimes_{R\pw{t}[y]} R\pw{t}[y^{\pm 1}]$. The main point (we omit the argument here) is that there is a quasi-isomorphism of $R\pw{t}$-coalgebras
\begin{equation}\label{eq: qiso-coalgebra}
    R\pw{t}[y^{\pm 1}] \otimes_{\fdr{\GG_m}} R\pw{t}[y^{\pm 1}] \simeq R\pw{t}\left[y_1^{\pm 1}, y_2^{\pm 1}, \frac{(y_1-y_2)^n_s}{n!_F}\right]_{n\geq 0} \simeq \co_{\GG_m^{\sharp,F} \times_{\spf R\pw{t}} \GG_m},
\end{equation}
where the tensor product on the left-hand side is derived, and both sides are implicitly $(p,t)$-adically completed. The final equivalence follows by noting that $(y_1-y_2)^n_s = y_2 (y_1 y_2^{-1} - 1)^n_s$, so adjoining $\frac{(y_1-y_2)^n_s}{n!_F}$ is equivalent to adjoining the $F$-divided powers $\frac{(y_1 y_2^{-1} - 1)^n_s}{n!_F}$.

In the case of the additive formal group law, \cref{eq: qiso-coalgebra} boils down to the $p$-complete equivalence
$$\Z_p[y^{\pm 1}] \otimes_{\dR_{\Z_p[y^{\pm 1}]/\Z_p}} \Z_p[y^{\pm 1}] \simeq \dR_{\Z_p[y^{\pm 1}]/\Z_p[y_1^{\pm 1}, y_2^{\pm 1}]} \simeq \Z_p[y_1^{\pm 1}, y_2^{\pm 1}]\pdb{y_1-y_2} \cong \co_{\GG_m^\sharp \times_{\spf \Z_p} \GG_m}$$
which arises via \cite[Proposition 8.5]{bhatt-ddr}.
In the case of the multiplicative formal group law, \cref{eq: qiso-coalgebra} was shown in \cite{pridham-q-dR}.

Given this, \cref{thm: BL-analogue} can be rephrased as the following two statements:
\begin{enumerate}
    \item The canonical map $(\GG_m)_{R\pw{t}} \to \GG_m^{F\dR}$ factors through the Frobenius $(\GG_m)_{R\pw{t}} \to (\GG_m^{(1)})_{R\pw{t}}$.
    \item The map $(\GG_m^{(1)})_{R\pw{t}} \to \GG_m^{F\dR}$ is surjective, and its kernel is isomorphic to the Cartier dual of the rescaled formal group $\hat{\GG}_t$. Here, the map $\hat{\GG}_t^\vee \to (\GG_m^{(1)})_{R\pw{t}}$ is Cartier dual to the homomorphism $p\ul{\Z} \to \hat{\GG}_t$ sending $p\mapsto \pdb{p}_F(t)$. In particular, there is an isomorphism 
    $$\GG_m^{F\dR} \simeq (\GG_m^{(1)})_{R\pw{t}}/ \hat{\GG}_t^\vee$$
    over $\spf R\pw{t}$.
\end{enumerate}
In the case of the multiplicative formal group law, this picture has been discussed in great detail in \cite{drinfeld-formal-group}. As we hope to explain in future work, the above two statements are natural consequences of the homotopy-theoretic perspective on the $F$-de Rham complex (see \cref{rmk: homotopy theory}). We also hope that \cref{thm: BL-analogue} might be useful in understanding \cref{conj: polynomial maps}, and, in particular, in understanding a ``stacky approach'' to $F$-de Rham cohomology (following \cite{apc, drinfeld-crys, drinfeld-prism, bhatt-f-gauge-lectures}). Indeed, the $F$-divided power scheme $\GG_m^{\sharp,F}$ is \textit{a priori} rather difficult to access algebro-geometrically, but \cref{thm: BL-analogue} says that it can be understood concretely in terms of $(\GG_m)_{R\pw{t}}$ and the Cartier dual of the rescaled formal group $\hat{\GG}_t$.
\end{remark}
Let us end by noting that the following is an immediate consequence of \cref{thm: BL-analogue}:
\begin{corollary}\label{cor: cartier dual of Gm-sharp}
Let $(\GG_m^{\sharp,F})^\vee$ denote the Cartier dual of $\GG_m^{\sharp, F}$. Then, there is a pushout square over $R\pw{t}$:
$$\xymatrix{
\ul{p\Z}_{R\pw{t}} \ar[d] \ar[r]^-{p \mapsto \pdb{p}_F(t)} & \hat{\GG}_t \ar[d] \\
\ul{\Z}_{R\pw{t}} \ar[r] & (\GG_m^{\sharp,F})^\vee.
}$$
\end{corollary}

\newpage
\bibliographystyle{alpha}
\bibliography{main}

\begin{thebibliography}{{Dev}23b}

\bibitem[AL20]{anschutz-le-bras}
J.~{Ansch\"{u}tz} and A.-C. {Le Bras}.
\newblock {The {$p$}-completed cyclotomic trace in degree 2}.
\newblock {\em {Ann. K-Theory}}, 5(3):539--580, 2020.

\bibitem[{Bha}00]{bhargava-factorial-AMM}
M.~{Bhargava}.
\newblock {The factorial function and generalizations}.
\newblock {\em {Amer. Math. Monthly}}, 107(9):783--799, 2000.

\bibitem[{Bha}12]{bhatt-ddr}
B.~{Bhatt}.
\newblock {$p$-adic derived de Rham cohomology}.
\newblock \url{https://arxiv.org/abs/1204.6560}, 2012.

\bibitem[{Bha}22]{bhatt-f-gauge-lectures}
B.~{Bhatt}.
\newblock {Prismatic F-gauges}.
\newblock {Lecture notes available at
  \url{https://www.math.ias.edu/~bhatt/teaching/mat549f22/lectures.pdf}}, 2022.

\bibitem[BL22]{apc}
B.~{Bhatt} and J.~{Lurie}.
\newblock {Absolute prismatic cohomology}.
\newblock \url{https://arxiv.org/abs/2201.06120}, 2022.

\bibitem[BMS18]{bms-i}
B.~{Bhatt}, M.~{Morrow}, and P.~{Scholze}.
\newblock Integral {$p$}-adic {H}odge theory.
\newblock {\em Publ. Math. Inst. Hautes \'{E}tudes Sci.}, 128:219--397, 2018.

\bibitem[BMS19]{bms-ii}
B.~{Bhatt}, M.~{Morrow}, and P.~{Scholze}.
\newblock Topological {H}ochschild homology and integral {$p$}-adic {H}odge
  theory.
\newblock {\em Publ. Math. Inst. Hautes \'{E}tudes Sci.}, 129:199--310, 2019.

\bibitem[BO78]{berthelot-ogus}
P.~{Berthelot} and A.~{Ogus}.
\newblock {\em Notes on crystalline cohomology}.
\newblock Princeton University Press, Princeton, N.J.; University of Tokyo
  Press, Tokyo, 1978.

\bibitem[BS19]{bhatt-scholze}
B.~{Bhatt} and P.~{Scholze}.
\newblock {Prisms and Prismatic Cohomology}.
\newblock \url{https://arxiv.org/abs/1905.08229}, 2019.

\bibitem[{Dev}23a]{coh-gr}
S.~{Devalapurkar}.
\newblock {Chromatic aberrations of geometric Satake over the regular locus}.
\newblock \url{https://arxiv.org/abs/2303.09432}, 2023.

\bibitem[{Dev}23b]{thh-xn}
S.~{Devalapurkar}.
\newblock {Topological Hochschild homology, truncated Brown-Peterson spectra,
  and a topological Sen operator}.
\newblock \url{https://arxiv.org/abs/2303.17344}, 2023.

\bibitem[DR23]{tp-Z}
S.~{Devalapurkar} and A.~{Raksit}.
\newblock {TBD}.
\newblock Forthcoming, 2023.

\bibitem[{Dri}18]{drinfeld-crys}
V.~{Drinfeld}.
\newblock {A stacky approach to crystals}.
\newblock \url{https://arxiv.org/abs/1810.11853}, 2018.

\bibitem[{Dri}21]{drinfeld-formal-group}
V.~{Drinfeld}.
\newblock {A 1-dimensional formal group over the prismatization of $\spf
  \Z_p$}.
\newblock \url{http://arxiv.org/abs/2107.11466}, 2021.

\bibitem[{Dri}22]{drinfeld-prism}
V.~{Drinfeld}.
\newblock {Prismatization}.
\newblock \url{https://arxiv.org/abs/2005.04746}, 2022.

\bibitem[{Eul}53]{original-euler}
L.~{Euler}.
\newblock {Consideratio quarumdam serierum, quae singularibus proprietatibus
  sunt praeditae}.
\newblock {\em {Novi Commentarii academiae scientiarum Petropolitanae}},
  3:86--108, 1753.
\newblock Available at
  \url{https://scholarlycommons.pacific.edu/euler-works/190/}.

\bibitem[{Haz}78]{hazewinkel}
M.~{Hazewinkel}.
\newblock {\em {Formal groups and applications}}, volume~78 of {\em {Pure and
  Applied Mathematics}}.
\newblock {Academic Press, Inc. [Harcourt Brace Jovanovich, Publishers], New
  York-London}, 1978.

\bibitem[HLN21]{l-theory-of-Z}
F.~{Hebestreit}, M.~{Land}, and T.~{Nikolaus}.
\newblock {On the homotopy type of {L}-spectra of the integers}.
\newblock {\em {J. Topol.}}, 14(1):183--214, 2021.

\bibitem[{Hon}70]{original-honda-fgl}
T.~{Honda}.
\newblock {On the theory of commutative formal groups}.
\newblock {\em {J. Math. Soc. Japan}}, 22:213--246, 1970.

\bibitem[HRW22]{even-filtr}
J.~{Hahn}, A.~{Raksit}, and D.~{Wilson}.
\newblock {A motivic filtration on the topological cyclic homology of
  commutative ring spectra}.
\newblock \url{https://arxiv.org/abs/2206.11208}, 2022.

\bibitem[{Jac}09]{original-jackson}
F.~{Jackson}.
\newblock {On $q$-Functions and a certain Difference Operator}.
\newblock {\em {Earth and Environmental Science Transactions of The Royal
  Society of Edinburgh }}, 46(2):253--281, 1909.

\bibitem[KC02]{quantum_calculus}
V.~{Kac} and P.~{Cheung}.
\newblock {\em Quantum Calculus}.
\newblock Universitext. Springer New York, 2002.

\bibitem[{Lon}17]{lonergan-steenrod}
G.~{Lonergan}.
\newblock {Steenrod Operators, the Coulomb Branch and the Frobenius Twist, I}.
\newblock \url{https://arxiv.org/abs/1712.03711}, 2017.

\bibitem[{Pri}19]{pridham-q-dR}
J.~P. {Pridham}.
\newblock {On {$q$}--de {R}ham cohomology via {$\Lambda $}-rings}.
\newblock {\em {Math. Ann.}}, 375(1-2):425--452, 2019.

\bibitem[{Rav}86]{green}
D.~{Ravenel}.
\newblock {\em {Complex cobordism and stable homotopy groups of spheres}}.
\newblock {Academic Press}, 1986.

\bibitem[Sas18]{mathoverflow-q-vandermonde}
Sasha.
\newblock {Is there a lift of the q-Vandermonde identity to some geometric
  (motivic) identity for Grassmannians over $F_q$?}
\newblock \url{https://mathoverflow.net/q/299582}, 2018.

\bibitem[{Sch}17]{scholze-q-def}
P.~{Scholze}.
\newblock Canonical {$q$}-deformations in arithmetic geometry.
\newblock {\em {Ann. Fac. Sci. Toulouse Math. (6)}}, 26(5):1163--1192, 2017.

\end{thebibliography}
\end{document}